\documentclass[]{article}
\usepackage{amsmath, amsthm, amssymb,enumerate}

\usepackage[numbers,sort&compress]{natbib}
\usepackage{tikz}
\usetikzlibrary{shapes,arrows,fit,calc,positioning,patterns}

\newtheorem{theorem}{Theorem}
\newtheorem{lemma}{Lemma}

\newtheorem{corollary}{Corollary}

\newtheorem{proposition}{Proposition}

\usepackage{graphicx}
\usepackage{hyperref}
\usepackage{xcolor}
\usepackage{comment}
\usepackage{algorithm}

\usepackage{tikz}
\usepackage{pgfplots,amsmath}
\usetikzlibrary{shapes,arrows,fit,calc,positioning,patterns,decorations.pathmorphing,decorations.pathreplacing}

\tikzset{ brokenrect/.style={

    append after command={

      \pgfextra{

      \path[draw,#1]

       decorate[decoration={zigzag,segment length=0.3em, amplitude=.7mm}]

       {(\tikzlastnode.north east)--(\tikzlastnode.south east)}      

        -- (\tikzlastnode.south west)|-cycle;

        }}}}

\tikzset{ brokenrect2/.style={

    append after command={

      \pgfextra{

      \path[draw,#1]

       decorate[decoration={zigzag,segment length=0.3em, amplitude=.7mm}]

       {(\tikzlastnode.north west)--(\tikzlastnode.south west)}      

        -- (\tikzlastnode.south east)|-cycle;

        }}}}
        
\tikzset{cross/.style={cross out, draw=black, minimum size=2*(#1-\pgflinewidth), inner sep=0pt, outer sep=0pt},
cross/.default={1pt}}

\newcommand{\bigO}{\mathcal{O}}
\newcommand{\NP}{\mathcal{NP}}
\newcommand{\Cs}{\sum C_j}
\newcommand{\Cmax}{C_{\max}}
\newcommand{\Copt}{C^{{\sc OPT}}}
\newcommand{\Copts}{C^{{\sc OPT}_s}}
\newcommand{\Coptf}{C^{{\sc OPT}_f}}
\newcommand{\Calg}{C^{{\sc ALG}}}
\newcommand{\Salg}{\sigma^{{\sc ALG}}}
\newcommand{\J}{\mathcal{J}}
\newcommand{\tpart}{{\sc 3-Partition}}

\renewcommand{\|}{\vert}

\begin{document}
\title{Approximation algorithms for coupled task scheduling minimizing the sum of completion times}
%\titlerunning{Approx coupled task scheduling $\Cs$}
%
\author{David Fischer\footnote{Institute of Algorithms and Complexity, Hamburg University of Technology, Blohmstraße 15, 21079 Hamburg, Germany, e-mail: da.fischer@tuhh.de.
} 
~and %\inst{1}
%\thanks{Supported by DFG grant MN 59/4-1.}\orcidID{0000-0001-8402-1818} \and
P\'{e}ter Gy\"{o}rgyi\footnote{Institute for Computer Science and Control, E\"{o}tv\"{o}s Lor\'{a}nd Research Network, Kende Str. 13-17., 1111 Budapest, Hungary, e-mail: gyorgyi.peter@sztaki.hu.
}
%\inst{2}
%\thanks{Supported by the National Research, Development and Innovation Office grant no. TKP2021-NKTA-01 and by the J\'anos Bolyai Research Scholarship of the Hungarian Academy of Sciences}\orcidID{0000-0002-2380-5528}
}
%
%\authorrunning{D. Fischer and P. Györgyi}
%\institute{Institute of Algorithms and Complexity, Hamburg University of Technology, Blohmstraße 15, 21079 Hamburg, Germany
%\email{da.fischer@tuhh.de}
%\and
%ELKH Institute for Computer Science and Control, Kende Str. 13-17., 1111 Budapest, Hungary
%\email{gyorgyi.peter@sztaki.hu}}
%
\maketitle              % typeset the header of the contribution
%
%\sloppy
%\begin{abstract}
%\textcolor{red}{The abstract should briefly summarize the contents of the paper in
%15--250 words.}
%
%%\keywords{Single machine scheduling \and Coupled task problem \and Approximation algorithms.}
%\end{abstract} 
%
%
%
\begin{abstract}
In this paper we consider the coupled task scheduling problem with exact delay times on a single machine with the objective of minimizing the total completion time of the jobs. 
We provide constant-factor approximation algorithms for several variants of this problem that are known to be $\NP$-hard, while also proving $\NP$-hardness for two variants whose complexity was unknown before. 
Using these results, together with constant-factor approximations for the makespan objective from the literature, we  also introduce the first results on bi-objective approximation in the coupled task setting.
\end{abstract}

\section{Introduction}\label{sec:intro}
The problem of scheduling coupled tasks with exact delays (CTP) was introduced by Shapiro~\cite{shapiro80} more than forty years ago. 
In this particular scheduling problem, each job has two separate tasks and a delay time. The goal is to schedule these tasks such that no tasks overlap, and the two tasks of a job are scheduled with exactly their given delay time in between them, while optimizing some objective function.
This problem has several practical applications, e.g., in pulsed radar systems, where one needs to receive the reflections of the transmitted pulses after a given period of time~\cite{elshafei04,farina80}, in improving the performance of submarine torpedoes~\cite{simonin11}, or in chemistry~\cite{ageev07}.

%Coupled task scheduling is also highly topical with regard to the Covid-19 pandemic.
%Most current vaccinations for this disease are to be administered at least twice for full protection, with a given delay time between the two shots depending on the employed vaccine, as recommended for example by the Center for Disease Control and Prevention~\cite{CDC22}.
%The process of administering these two vaccination shots to any given number of people can be modeled as a coupled task scheduling problem, by creating a job for each person, and defining the two respective vaccination shots as the two tasks of each job, with the recommended delay time in between.

Research interest in the coupled task problem is strongly increasing in recent years, see Khatami et al.~\cite{khatami19} for a current, detailed overview of the topic. This research focuses mainly on variants of the general CTP with some additional restrictions on the job properties, and mainly tries to optimize the makespan~\cite{khatami19}.
Coupled task problems are often $\NP$-hard even in very special cases, but polynomial-time approximation algorithms with constant approximation factors have been developed for a number of them~\cite{ageev06,ageev07,ageev2016}.
Other objective functions have virtually not been considered though, until recently, when Chen and Zhang~\cite{chen2021} drew an almost full complexity picture for the problems of single-machine scheduling of coupled tasks, with the objective of minimizing the total sum of job completion times.
However, they did not give any approximation algorithms for $\NP$-hard CTP variants with this particular objective function. 
We fill this gap by giving a number of constant-factor approximation algorithms for most of these CTP variants. Additionally, we introduce two new, interesting variants, which we also prove to be $\NP$-hard, and also approximate one of these with a constant factor.

%This paper considers CTP in a single machine environment, which means we can schedule at most one task at a time. 
Formally, we are given a set of $n$ jobs $\J=\{1, 2,\ldots,n\}$, where each job $j$ has two tasks: $a_j$ and $b_j$. 
We call $a_j$ the first task, and $b_j$ the second task of job $j$.
In order to simplify our notations, we will denote the processing time of these tasks also by $a_j$ and $b_j$; the meaning of these notations will be clear from context. The sum $(a_j + b_j)$ is then called the total processing time of a job $j$.
These tasks have to be scheduled on a single machine with a given delay time~$L_j$ in-between, which means if the machine completes $a_j$ at some time point $t$, then we have to schedule $b_j$ to start exactly at $t+L_j$. Preemption is not allowed.
Note that it is possible to schedule other tasks on the machine during this delay time, but the tasks themselves cannot overlap.
Our objective is to find a feasible schedule $\sigma$ that minimizes the total of job completion times, where a feasible schedule is defined as a schedule that fulfills all of the requirements above. Such a $\sigma$ is then called optimal schedule or optimal solution for the CTP instance.
For a given schedule $\sigma$, the starting time $S_j$ of $j$ is the starting time of $a_j$, while the completion time $C_j$ of $j$ is the completion time of $b_j$.
%As we only look at CTP with the objective of minimizing the total completion times, we call this problem "CTP" from now on for simplicity reasons.
An example of CTP is visualized in \autoref{fig:exampleCTP}.
For a schedule $\sigma$, a \emph{gap} is a period between time points $t_1$ and $t_2$ such that the machine is idle between~$t_1$ and $t_2$ and busy at both $t_1$ and $t_2$. The length of a gap is the length of this time window. A partial schedule $\sigma^p$ is a schedule for a subset of the jobs~$\J$.
\begin{figure}
\centering

\begin{tikzpicture}

\def\ox{0} 
\def\oy{0} 
\coordinate(o) at (\ox,\oy);

%axis
\def\tl{10.0}
\draw [-latex](\ox,\oy) node[above left]{} -- (\ox+\tl,\oy) node[above,font=\small]{$t$};

%definitions for jobs
\def\pi{0.5}
\tikzstyle{mystyle}=[draw, minimum height=0.5cm,rectangle, inner sep=0pt,font=\scriptsize]
\tikzstyle{mystyle2}=[draw = none, minimum height=0.25cm,rectangle, inner sep=0pt,font=\scriptsize]

\draw (0,0) -- (0,-0.2) node[below] {\tiny $S_1$};
\draw (\pi,0) -- (\pi,-0.2) node[below] {\tiny $S_2$};
\draw (2.5,0) -- (2.5,-0.2) node[below] {\tiny $S_3$};
\draw (7.5,0) -- (7.5,-0.2) node[below] {\tiny $C_1$};
\draw (4.75,0) -- (4.75,-0.2) node[below] {\tiny $C_2$};
\draw (9.5,0) -- (9.5,-0.2) node[below] {\tiny $C_3$};

\draw [<->] (\pi,0.65)--node[above]{\small $L_1$}(\pi+5,0.65); 

%jobs
\node(b1) [above right=-0.01cm and -0.01cm of o,mystyle, minimum width=\pi cm,pattern=north west lines, pattern color=red]{};
\node(b1_t) [mystyle2, fill = white] at (b1.center) {$a_1$};
\node(b2) [right=5cm of b1,mystyle, minimum width=2 cm,pattern=north west lines, pattern color=red]{};
\node(b2_t) [mystyle2, fill = white] at (b2.center) {$b_1$};
\node(b3) [right=0cm of b1,mystyle, minimum width=1.5 cm,pattern=north east lines, pattern color=green]{};
\node(b3_t) [mystyle2, fill = white] at (b3.center) {$a_2$};
\node(b4) [right=2cm of b3,mystyle, minimum width=0.75 cm,pattern=north east lines, pattern color=green]{};
\node(b4_t) [mystyle2, fill = white] at (b4.center) {$b_2$};
\node(b5) [right=0.5cm of b3,mystyle, minimum width=1 cm,pattern=vertical lines, pattern color=yellow]{};
\node(b5_t) [mystyle2, fill = white] at (b5.center) {$a_3$};
\node(b6) [right=1cm of b2,mystyle, minimum width=1 cm,pattern=vertical lines, pattern color=yellow]{};
\node(b6_t) [mystyle2, fill = white] at (b6.center) {$b_3$};
\end{tikzpicture}

\caption{An example for a feasible solution for an instance of CTP with $n = 3$. The patterns are matching for the two tasks of each job $j$. For simplicity, the delay time is only visualized for job $1$. The total completion time of the solution is $C_1 + C_2 + C_3$.}
\label{fig:exampleCTP}
\end{figure}
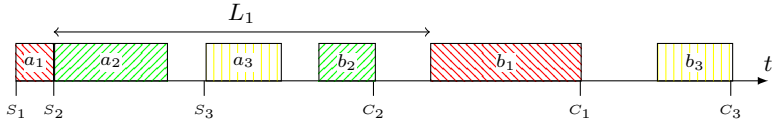

We say a job is the $j^{\textnormal{th}}$ finishing job in a schedule if its second task is scheduled after the second tasks of exactly $j-1$ many jobs have finished.
We say a job is the $j^{\textnormal{th}}$ starting job in a schedule if its first task is scheduled after the first tasks of exactly $j-1$ many jobs have been scheduled.
Let $\Copt$ be the sum of completion times of an optimal solution. 
For a fixed optimal schedule, let $\Coptf_j$ be the completion time of the $j^{\textnormal{th}}$ finishing job in that schedule.
Analogously, let $\Copts_j$ be the completion time of the $j^{\textnormal{th}}$ starting job in that schedule.
Observe that $\Copt = \sum_{j=1}^{n} \Coptf_j=\sum_{j=1}^{n} \Copts_j$.
In any proof in this work, we denote the sum of completion times of the solution  produced by the currently used algorithm as $\Calg$.
In the schedule $\Salg$ created by this algorithm, $S_j$ and $C_j$ denote the starting time and  completion time of job~$j$, respectively.

Throughout this paper, we will use the classic $\alpha \| \beta\|\gamma$ notation system of Graham et al.~\cite{graham79}, with $\alpha$ representing the machine environment, $\beta$ representing the characteristics of the jobs, and $\gamma$ representing the objective function. 
$1\|(a_j, L_j, b_j)\|\Cs$ then denotes the general CTP for minimizing the sum of completion times, where each job $j$ consists of a pair of tasks of processing times $a_j$ and $b_j$, respectively, with an exact time delay $L_j$ between the completion time of its first task and the start time of its second task. 
As we also look at more restricted variants of CTP, we fix some naming conventions to easily express these restrictions in Graham notation. 
If in a restricted CTP environment, some task of the jobs are fixed or even constant for each job $j$, we denote this task without the subscript '$j$', or by the specific constant value; e.g. CTP, where the delay time $L_j$ is fixed to some $L$ for all jobs $j$ is denoted as $1\|(a_j, L, b_j)\|\Cs$. 
If in a restricted CTP environment, some tasks of the same job always have the same value, we denote them by $p$ instead of their usual descriptor; e.g., CTP where the first and second task of each job $j$ have the same processing time ($a_j = b_j, \forall j$) is denoted as $1\|(p_j, L_j, p_j)\|\Cs$.
This is in line with the standard notation for the coupled task scheduling problems, as seen for example in Chen and Zhang~\cite{chen2021}.
Another way to restrict CTP is to fix the processing sequence of the first tasks of the jobs, this is indicated by~$\pi_a$ in the $\beta$-field of the Graham notation.

In this paper we extend the complexity results of Chen and Zhang~\cite{chen2021} and Kubiak~\cite{kubiak22} by proving the strong $\NP$-hardness of $1\|(p_j, L, p_j)\|\Cs$ and $1\|(1,L_j,1, \pi_a)\|\sum C_j$. To achieve the former, we first prove strong $\NP$-hardness of the corresponding makespan variant $1\|(p_j, L, p_j)\|\Cmax$, strengthening a result by Ageev and Ivanov~\cite{ageev2016}, who prove weak $\NP$-hardness of this problem.
We also give constant-factor approximations for most CTP variants in a single machine environment with the sum of completion times objective function, see \autoref{fig:overview}.
The existence of a constant-factor approximation algorithm for the variants $1\|(a_j, L_j, b_j)\|\Cs$, $1\|(a, L_j, b_j)\|\Cs$, $1\|(a_j, L_j, b)\|\Cs$, and $1\|(p_j, L_j, p_j)\|\Cs$ is still open (see the upper part of the figure).

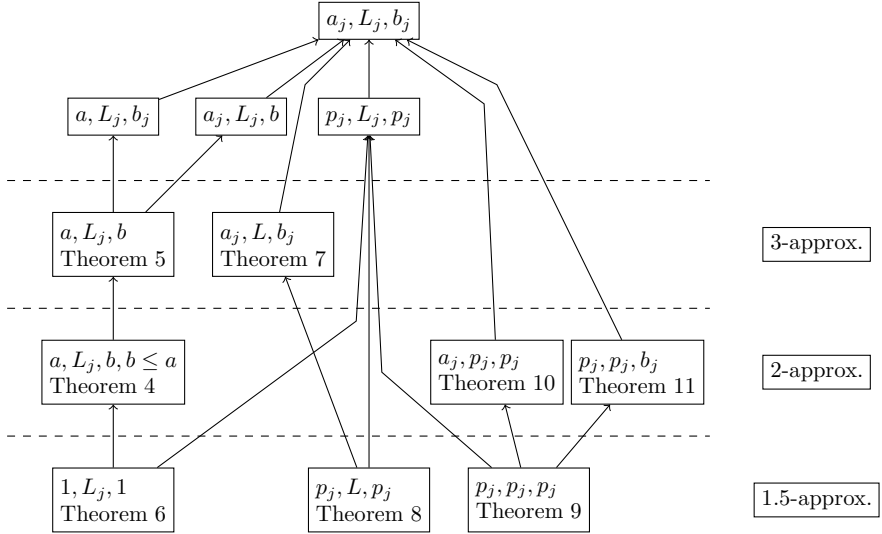
\begin{figure}[h!]
\centering

\begin{tikzpicture}[scale=0.85, every node/.style={transform shape}]
\tikzstyle{style_a}=[minimum height=1cm, shape=rectangle,draw=black, align=left]
    \node[shape=rectangle,draw=black] (A) at (2,-0.5) {$a_j, L_j, b_j$};
    
    \node[shape=rectangle,draw=black] (B) at (-2,-2) {$a, L_j, b_j$};
    \node[shape=rectangle,draw=black] (C) at (0,-2) {$a_j, L_j, b$};
    \node[shape=rectangle,draw=black] (D) at (2,-2) {$p_j, L_j, p_j$};
    
    \node[shape=rectangle,draw=black] (3) at (9,-4) {$3$-approx.};
    \node[style_a] (E) at (0.5,-4) {$a_j, L, b_j$ 
    \\ \autoref{lemma:ajLbj}};
    %\node[shape=rectangle,draw=black] (F) at (0.66,-4) {$a, L, b_j$ \ref{lemma:ajLbj}};
    %\node[shape=rectangle,draw=black] (G) at (6,-4) {$a_j, L, b$ \ref{lemma:ajLbj}};
    \node[style_a] (H) at (-2,-4) {$a, L_j, b$ \\ \autoref{lemma:aLjbgreaterb}};
    
    \node[shape=rectangle,draw=black] (2) at (9,-6) {$2$-approx.};
    %\node[shape=rectangle,draw=black, align=left] (I) at (0.66,-6) {$p, L_j, p$\\ \autoref{lemma:aLjbsmallerb}};
    \node[style_a] (J) at (-2,-6) {$a, L_j, b, b \leq a$ \\ \autoref{lemma:aLjbsmallerb}};
    \node[style_a] (K) at (4,-6) {$a_j, p_j, p_j$ \\ \autoref{thm:ajpjpj}};
    \node[style_a] (L) at (6.2,-6) {$p_j, p_j, b_j$ \\ \autoref{thm:pjpjbj}};
    
    \node[shape=rectangle,draw=black] (1.5) at (9,-8) {$1.5$-approx.};
    \node[style_a] (M) at (2,-8) {$p_j, L, p_j$ \\ \autoref{lemma:pjLpj}};
    \node[style_a] (N) at (-2,-8) {$1, L_j, 1$ \\ \autoref{lemma:1Lj1}};
    \node[style_a] (O) at (4.5,-8) {$p_j, p_j, p_j$ \\ \autoref{thm:pjpjpj}};
    
	\path [-] (-3.66,-3) edge[dashed] (7.33,-3);
	\path [-] (-3.66,-5) edge[dashed] (7.33,-5);
	\path [-] (-3.66,-7) edge[dashed] (7.33,-7);         
    
    \path [->] (B) edge (A);
    \path [->] (C) edge (A);
    \path [->] (D) edge (A);
    
    \path [->] (H) edge (B);
    \path [->] (H) edge (C);
    \draw [->] (E) -- (1,-1.5) -- (A);
    %\path [->] (F) edge (E);
    %\path [->] (G) edge (E);
    
    %\path [->] (I) edge (J);
    \path [->] (J) edge (H);
    \draw [->] (K) --(3.8,-1.8) -- (A);
    \draw [->] (L) -- (4,-1.5) --(A);
    
    \path [->] (N) edge (J);
    \path [->] (M) edge (E);
    \path [->] (O) edge (K);
    \path [->] (O) edge (L);
    
    \path [->] (M) edge (D);
    %\path [->,bend right] (N) edge (D);
    %\path [->,bend left] (O) edge (D);
    \draw[->] (N) -- (1.8,-5.2)-- (D);
    \draw[->] (O) --(2.2,-6) --(D);
\end{tikzpicture}

\caption{Overview of our approximation results for different variants of $1\|(a_j,L_j,b_j)\|\sum C_j$. The variants, identified by their special constraints, are grouped into layers of equal approximation factors, with the respective Theorem proving this approximation factor linked next to it. A directed edge from variant "A" to variant "B" indicates that "B" is a generalization of "A".}
\label{fig:overview}
\end{figure}

We also look at bi-objective optimization for CTP with both the makespan and the sum of completion times objectives, under the goal of minimizing both objectives without prioritization.
For this, we use the concept of $(\rho_1, \dots, \rho_z)$-approximation, as introduced by Jiang et al.~\cite{jiang22} for simultaneously minimizing~$z$ objectives, where $\rho_i$ is the approximation factor of the $i^{\textnormal{th}}$ objective function to be minimized, for $i = 1, \dots, z$.
This concept is a generalization of the bi-objective $(\rho_1, \rho_2)$-approximation of scheduling problems minimizing makespan and sum of completion times, as described first by Stein and Wein~\cite{stein97}.
As far as we are aware, there are no results of minimizing the two objectives makespan and sum of completion times simultaneously in a coupled task setting, even though this topic is well researched in other scheduling environments.
We start to close this gap by directly using results of Stein and Wein~\cite{stein97}, together with constant-factor approximations of the makespan~\cite{ageev06, ageev07, ageev2016} and our approximation results on the sum of completion times objective, to give a number of $(\rho_1, \rho_2)$-approximation results for this problem.
The general bi-objective CTP is denoted as $1\|(a_j, L_j, b_j)\|(\Cmax, \Cs)$ in $\alpha\|\beta\|\gamma$ notation system, with the $\beta$-field following the previously discussed naming conventions depending on the considered variant's restrictions.

This work is structured as follows.
We first give a brief literature review of the topic in \autoref{sec:lit_rev}.
%This is  followed by a formal problem statement    in \autoref{subsec:problem_statement}.
We present our complexity results in \autoref{sec:complex}.
%In \autoref{sec:lbopt}, we give the different lower bounds on optimal solutions of the problem used for the analysis of the aforementioned approximation factors.
Our approximation results are stated in \autoref{sec:main}. There, we give detailed descriptions and run time analyses of our algorithms, as well as proofs on approximation factors for problem variants whose instances can be solved by these algorithms.
We then use these results in \autoref{sec:biobj} to give $(\rho_1, \rho_2)$-approximations for the bi-objective CTP with both the makespan and sum of completion time objectives.
Finally, we give concluding remarks and an outlook on future research in \autoref{sec:concl}.

\section{Literature review}
\label{sec:lit_rev}
Research on coupled task scheduling on a single machine began when Shapiro~\cite{shapiro80} proved the $\NP$-hardness of the general problem $1\|(a_j, L_j, b_j)\|\Cmax$, where both tasks, as well as the delay time between them, can be different for each job, and the makespan is to be minimized.

In subsequent years, the $\NP$-hardness was also shown for more restricted variants of this problem, specifically $1\|(p_j, p_j, p_j)\|\Cmax$, $1\|(a_j, L, b)\|\Cmax$, $1\|(a, L, b_j)\|\Cmax$ and $1\|(p, L_j, p)\|\Cmax$ by Orman and Potts~\cite{orman97}. 
%These directly imply the $\NP$-hardness of problem variants with restrictions "in-between" these variants and the general $1\|(a_j, L_j, b_j)\|\Cmax$ problem, e.g. $1\|(a_j, L, b_j)\|\Cmax$.
Some CTP variants minimizing the makespan are $\NP$-hard even when the processing times of all jobs are fixed to $1$, as shown by Yu et al.~\cite{yu04} for $1\|(1, L_j, 1)\|\Cmax$. Condotta and Shakhlevich~\cite{condotta12} showed $\NP$-hardness for the even more restricted variant  $1\|(1, L_j, 1, \pi_a)\|\Cmax$, where $\pi_a$ indicates a fixed processing sequence for the first tasks of all jobs.

For most of these problems, polynomial-time constant-factor approximation algorithms have been developed. Ageev and Kononov~\cite{ageev06} give such algorithms, as well as inapproximability bounds, for the general $1\|(a_j, L_j, b_j)\|\Cmax$ problem, and the restricted variants $1\|(a_j, L_j, b_j, a_j \leq b_j)\|\Cmax$, $1\|(a_j, L_j, b_j, a_j \geq b_j)\|\Cmax$, and $1\|(p_j, L_j, p_j)\|\Cmax$.
Related to this work, Ageev and Baburin~\cite{ageev07} give an approximation algorithm for the  $1\|(1, L_j, 1)\|\Cmax$ variant.
Additionally, Ageev and Ivanov~\cite{ageev2016} give approximation algorithms and inapproximability bounds for  $1\|(a_j, L, b_j)\|\Cmax$, $1\|(a_j, L, b_j, a_j \leq b_j)\|\Cmax$ and $1\|(p_j, L, p_j)\|\Cmax$.

For other restricted variants, polynomial-time algorithms do exist. This was shown by Orman and Potts~\cite{orman97} for the variants $1\|(p, p, b_j)\|\Cmax$, $1\|(a_j, p, p)\|\Cmax$, and $1\|(p, L, p)\|\Cmax$, as well as Hwang and Lin~\cite{hwang11} for the variant $1\|(p_j, p_j, p_j), \text{fjs}\|\Cmax$, where "fjs" denotes that the sequence of jobs in the schedule is fixed.

Research interest in the topic of coupled task scheduling remained high also in the last years.
Békési et al.~\cite{bekesi22} recently introduced and gave a constant-factor approximation algorithm for the novel problem variant $1\|(1, L_j, 1), L_j \in \{L_1, L_2\}\|\Cmax$, where there are only two different delay times in an instance; the complexity status of this variant is still unknown.
Khatami and Selhipour~\cite{khatami21_1} tackle the coupled task scheduling problem differently, giving upper and lower bounds on the solution through different procedures, and proposing a binary heuristic search algorithm for CTP.
The same authors give optimal solutions under certain conditions, and a general heuristic for the problem variant with fixed first tasks and delay times, but time-dependent processing times for the second tasks~\cite{khatami21_2}.
Bessy and Giroudeau~\cite{bessy19} investigate CTP under parameterized complexity, with the considered parameter $k$ relating to the question if $k$ coupled tasks have a completion time before a fixed due date.

Interest is also high in scheduling coupled tasks in \emph{2-machine flow shop} environments, denoted by $F2$ in the machine environment notation. 
For scheduling coupled tasks in this environment, we are given two machines instead of one, and each of the two task of one job is additionally assigned one of these two machines to be processed on.
$\NP$-hardness is shown for a number of flow shop problems minimizing the makespan, e.g., for $F2\|(1, L_j, 1)\|\Cmax$ by Yu et al.~\cite{yu04}, but $\NP$-hardness is also known for variants minimizing the total completion time, e.g., $F2\|(a_j, L, b_j)\|\Cs$, as shown by Leung et al.~\cite{leung07}. 
Several flow shop problem variants minimizing the total completion time are also polynomial solvable, as proven by Leung et al.~\cite{leung07} and Huo et al.~\cite{huo09}.

All of the mentioned literature for scheduling coupled tasks on a single machine only considers the objective of minimizing the makespan though, and, as Khatami et al.~\cite{khatami19} note in their survey of CTP, “there has been no published research investigating the single-machine setting with an objective function other than the makespan, except for those in the cyclic setting.”
This task is finally tackled by Chen and Zhang~\cite{chen2021}, who draw a nearly full complexity picture of problem  of minimizing the total of job completion times.
%They do this by either showing $\NP$-hardness, or giving polynomial-time algorithms for the most important restricted variants of CTP with the objective of minimizing the total sum of job completion times.
However, they do not give any approximation algorithms for problem variants they prove to be $\NP$-hard.
Recently, Kubiak~\cite{kubiak22} slightly extended these complexity results by proving $\NP$-hardness of $1\|(1,L_j,1)\|\Cs$ and $1\|\langle 1,L_j,1 \rangle\|\Cs$. 
In the latter problem variant, the delay time between the two tasks does not have to be exactly, but at most $L_j$.

In scheduling theory, there is also a great interest in bi-objective and multi-objective optimization.
Here, instead of trying to optimize just one objective function in a given problem setting, one aims to optimize two or more objective functions at the same time, see Deb~\cite{deb14} or Hoogeveen~\cite{hoogeveen05} for an overview.
Since, until recently, virtually only the makespan objective has been considered for coupled tasks scheduling problems, we do not know any such results in the couple task environment.
This is not true for other scheduling environments though, where especially bi-objective optimization is intensively researched, particularly for the two objectives of minimizing makespan and sum of completion times.
Here, many approaches focus on establishing a trade-off relationship between the two competing objectives, either by \emph{Pareto optimization} (finding one or all Pareto optimal solutions) or \emph{simultaneous optimization} (minimizing all objectives without prioritization)~\cite{jiang22}.
Since these problems are generally $\NP$-hard (see e.g. Hoogeeven~\cite{hoogeveen05}), approximation is a popular method for both mentioned approaches.
Angel et al.~\cite{angel03} give fully polynomial time approximation schemes for the Pareto curve of single-machine batching problems and parallel machine scheduling problems on the two objectives.
Bampis and Knonov~\cite{bampis05} consider $(\rho_1, \rho_2)$-approximations of the two objectives for scheduling problems with communication delays.
A very recent work by Jiang et al.~\cite{jiang22} is concerned with $(\rho_1, \rho_2)$-approximations for scheduling on parallel machines, with different approximation ratios for different fixed numbers of 	machines.

%\subsection{Problem statement}\label{subsec:problem_statement}

%\section{Preliminaries}
%\label{sec:prelim}

\section{Complexity results}
\label{sec:complex}

In this section we prove that both $1\|(p_j,L,p_j)\|\Cs$ and $1\|(1,L_j,1, \pi_a)\|\Cs$ are strongly $\NP$-hard. We use reductions from corresponding makespan minimization problems for both results. As we need strong $\NP$-hardness of the corresponding problems for both our reductions, we additionally prove strong $\NP$-hardness of $1\|(p_j,L,p_j)\|\Cmax$; weak $\NP$-hardness was already proven for this problem by Ageev and Ivanov~\cite{ageev2016}.

\begin{theorem}
\label{thm:NPpjLpjCmax}
$1\|(p_j,L,p_j)\|\Cmax$ is strongly $\NP$-hard.
\end{theorem}

\begin{proof}
We reduce the well-known strongly $\NP$-hard problem \tpart~\cite{garey79} to $1\|(p_j,L,p_j)\|\Cmax$. The reduction is similar to the idea by Ageev and Ivanov~\cite{ageev2016} for reducing the weakly $\NP$-hard {\sc Partition} problem to $1\|(p_j,L,p_j)\|\Cmax$. First, let us formally state the \tpart{} problem.

\tpart{}

\textbf{Instance:} A set $Q = \{1, \dots, 3q\}$, and for each element $i \in Q$, a corresponding positive integer $e_i$ such that $\sum_{i \in Q} e_i = qE$, for some positive integer $E$, and $E/4 < e_i < E/2$.

\textbf{Question:} Does the set $Q$ partition into $q$ disjoint subsets $Q_1, \dots, Q_q$ such that $\sum_{i \in Q_j} e_i = E$, for $j = 1, \dots, q$?

\tpart{} remains strongly $\NP$-hard even if we assume $q$ is even. 
Consider an instance $I$ of \tpart{} where $q$ is even. 
We define an instance $I'$ of $1\|(p_j,L,p_j)\|\Cmax$ with $4q$ jobs as follows: 
%From now on, let $q$ be an even number (otherwise, simply introduce $3$ dummy jobs with $e_i = E/3$ and set $q := q+1$). 
\begin{itemize}
\item jobs $i = 1, \dots, 3q$ have $p_i = e_i$ and $L = R + E$ (small jobs),
\item jobs $i = 3q+1, \dots, 4q$ have $p_i = R$ and $L = R + E$ (large jobs),
\end{itemize}
for some $R > 3qE$. We prove there is a solution for $I$ if and only if there is a solution for $I'$ with makespan $\Cmax \leq z$, with $z := q(3E + 2R)$. The theorem follows from this statement.

\begin{enumerate}
\item
Assume there is a solution for $I$. Then there exist $Q_1, \dots, Q_q$, such that $\sum_{i \in Q_j} e_i = E$, for $j=1,\dots,q$. In this case, we create a schedule $\sigma$ for $I'$ with makespan at most $z$.
We schedule the jobs in blocks.
Let $B$ be an arbitrary block.
It consists of $2$ large jobs ($i$ and $i'$) and $6$ small jobs, which corresponds to items from some $Q_j$ and $Q_{j'}$.
Let $(j_1, j_2, j_3)$ and $(j'_1, j'_2, j'_3$) be the small jobs corresponding to the items in $Q_j$ and $Q_{j'}$, respectively.
Assume that $p_{j_1}\geq p_{j_2}\geq p_{j_3}$ and  $p_{j'_1}\leq p_{j'_2}\leq p_{j'_3}$.

%Let $\sigma^p$ be the partial schedule for $B$. 
We schedule $a_{i'}$ directly before $b_{i}$. 
We schedule $b_{j_1}, b_{j_2}$ and $b_{j_3}$ in this order in the gap between $a_i$ and $a_{i'}$ and $a_{j'_1},a_{j'_2}$ and $a_{j'_3}$ in this order in the gap between $b_i$ and $b_{i'}$, see \autoref{fig:nphardcmax}.
Observe that the job tasks do not intersect, because  the length of the gap between $a_i$ and $a_{i'}$ as well as the gap  between $b_i$ and $b_{i'}$ is exactly $E$.
%Since $a_{i} = b_{i} = a_{i'} = b_{i'} = R$, this does not create any intersection, but instead creates a schedule of the two large jobs with a gap of size $E$ between both $a_i$ and $a_{i'}$, as well as $b_i$ and $b_{i'}$, see \autoref{fig:nphardcmax}. We then take some $Q_j, Q_{j'}$ whose corresponding small jobs have not yet been scheduled in another block, and first schedule the second tasks of $Q_j$ in non-increasing order of their size in the gap between $a_i$ and $a_{i'}$, with their first tasks scheduled such that they finish $R + E$ time before the corresponding second tasks start. Secondly, we schedule the first tasks of the jobs in $Q_{j'}$ in non-decreasing order of their size in the gap between $b_i$ and $b_{i'}$, and their second tasks after the delay time $R + E$. Since $\sum_{x \in Q_j} b_x = \sum_{x' \in Q_{j'}} a_{x'} = E$, this always creates a feasible schedule, see \autoref{fig:nphardcmax}. Also, because $e_i \leq E$ for all $i$, the makespan of the tasks scheduled before $a_i$ and after $b_{i'}$ is at most $2E$ each. The makespan of all other tasks in this block is exactly $4R + 2E$. 
The length of block $B$ is  at most $4R + 6E$, because $a_{j_1}$ starts at most $2E$ before $a_i$, and $b_{j'_3}$ completes at most $2E$ after $b_{i'}$. 
We create such blocks for all jobs, resulting in $q/2$ blocks in total, and schedule these blocks directly one after another. 
Thus, the resulting schedule $\sigma$ has makespan $\Cmax \leq q/2 (4R + 6E) = z$.

\item Now assume that the instance $I'$ has a schedule $\sigma$ with makespan $\Cmax \leq z$. 
Due to the fixed delay times, the order is the same for the first and the second tasks in $\sigma$.
Consider an arbitrary large job $i$.
There has to be exactly one task of another large job $i'$ between $a_i$ and $b_i$.
There cannot be more than one task of a large job between $a_i$ and $b_i$. If there was no task of a large job between $a_i$ and $b_i$, the makespan of $\sigma$ would be larger than $z$, as the total processing time of the large jobs is $2qR$, and $L=R+E$, resulting in a minimum makespan of $\Cmax\geq 2qR+(R+E)> q(2R+3E)=z$, due to $R>3qE$.
Observe that if a task of job $i'$ is scheduled in the delay time of $i$, then a task of $i$ is scheduled in the delay time of $i'$.
This means we can partition the large jobs into pairs where the jobs within each pair are interleaved in the above way.

Consider an arbitrary pair of large jobs $(i,i')$ and assume that $S_i<S_{i'}$. 
There cannot be any task of a small job scheduled between $a_{i'}$ and $b_i$, as a first task of a small job would imply an intersection of its second task with $b_{i'}$, and a second task of a small job would imply an intersection of its first task with $a_i$, due to the fixed delay times.
Consider an arbitrary small job $j$. 
If none of its tasks is scheduled in the gap between $a_i$ and $b_{i'}$ of any pair of large jobs $(i,i')$, then there is also no task of a large job scheduled in the delay time of $j$ due to the fixed delay times.
This implies the makespan of $\sigma$ is at least $2qR+2e_j+R+E>z$, because the total processing time of the large jobs together with $j$ is $2qR+2e_j$, and the delay time of $j$ is $R+E$.

Therefore, for any small job $j$, there is a pair of large jobs $(i,i')$, where at least one of the tasks of $j$ is scheduled in the gap between $a_i$ and $a_{i'}$ or in the gap between $b_i$ and $b_{i'}$.
For any pair of large jobs the length of the gap between $a_i$ and $a_{i'}$ is at most $E$, the same holds for the length of the gap between $b_i$ and~¸$b_{i'}$.
Thus, the total length of these gaps is at most $q/2\cdot(2E)=qE$, which is exactly half of the total processing time of the small jobs.
Therefore, the length of each such gap must be exactly $E$, and they must be completely filled with tasks of small jobs.
As $E/4 < e_j < E/2$, there are always exactly three tasks of small jobs in each such gap. 
This partitions the small jobs into $q$ sets $Q_j$ of $3$ jobs each, with $\sum_{x \in Q_j} e_x = E$ for each $j = 1, \dots, q$, which gives us a feasible solution for the \tpart{} instance $I$.
\end{enumerate}

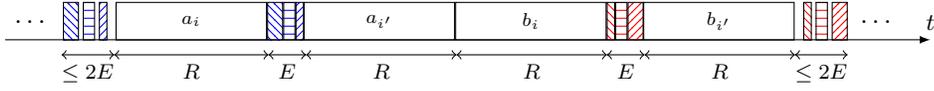
\begin{figure}
\centering
\begin{tikzpicture}
\def\ox{0} 
\def\oy{0} 
\coordinate(o) at (\ox,\oy);

%axis
\def\tl{12.3}
\draw [-latex](\ox,\oy) node[above left]{} -- (\ox+\tl,\oy) node[above,font=\small]{$t$};

%definitions for jobs
\def\pi{0.5}
\tikzstyle{mystyle}=[draw, minimum height=0.5cm,rectangle, inner sep=0pt,font=\scriptsize]
\tikzstyle{mystyle}=[draw, minimum height=0.5cm,rectangle, inner sep=0pt,font=\scriptsize]

\draw [<->]  (0.75,-0.2)--node[below] {\footnotesize $\leq 2 E$} (1.45,-0.2);
\draw [<->]  (1.45,-0.2)--node[below] {\footnotesize $R$} (3.5,-0.2);
\draw [<->]  (3.5,-0.2)--node[below] {\footnotesize $E$} (4,-0.2);
\draw [<->]  (4,-0.2)--node[below] {\footnotesize $R$} (6,-0.2); 
\draw [<->]  (6,-0.2)--node[below] {\footnotesize $R$} (8,-0.2);
\draw [<->]  (8,-0.2)--node[below] {\footnotesize $E$} (8.5,-0.2);
\draw [<->]  (8.5,-0.2)--node[below] {\footnotesize $R$} (10.5,-0.2);
\draw [<->]  (10.5,-0.2)--node[below] {\footnotesize $\leq 2 E$} (11.2,-0.2);
%\draw [decorate,decoration={brace,amplitude=3pt},xshift=0pt,yshift=0pt]
%(2.55cm,0cm) -- (0cm,0cm) node [below=0.1cm,black,midway]{\small  $\sigma$};

%jobs

\node(d1) [above right=-0.01cm and 0cm of o, minimum width=\pi cm, minimum height= \pi cm]{$\dots$};

\node(h1) [right=0.75cm of d1,mystyle, minimum width=2 cm]{$a_{i}$};
\node(h2) [right=2.5cm of h1,mystyle, minimum width=2 cm]{$b_{i}$};
\node(h3) [left=0cm of h2,mystyle, minimum width=2 cm]{$a_{i'}$};
\node(h4) [right=2.5cm of h3,mystyle, minimum width=2 cm]{$b_{i'}$};

\node(b1) [right=0cm of h1,mystyle,pattern=north west lines, pattern color=blue, minimum width=0.2 cm]{};
\node(b2) [right=0cm of b1,mystyle,pattern=horizontal lines, pattern color=blue, minimum width=0.15 cm]{};
\node(b3) [right=0cm of b2,mystyle,pattern=north east lines, pattern color=blue, minimum width=0.1 cm]{};
\node(b4) [left=2.5cm of b1,mystyle,pattern=north west lines, pattern color=blue, minimum width=0.2 cm]{};
\node(b5) [left=2.5cm of b2,mystyle,pattern=horizontal lines, pattern color=blue, minimum width=0.15 cm]{};
\node(b6) [left=2.5cm of b3,mystyle,pattern=north east lines, pattern color=blue, minimum width=0.1 cm]{};

\node(b7) [right=0cm of h2,mystyle,pattern=north west lines, pattern color=red, minimum width=0.1 cm]{};
\node(b8) [right=0cm of b7,mystyle,pattern=horizontal lines, pattern color=red, minimum width=0.15 cm]{};
\node(b9) [right=0cm of b8,mystyle,pattern=north east lines, pattern color=red, minimum width=0.2 cm]{};
\node(b10) [right=2.5cm of b7,mystyle,pattern=north west lines, pattern color=red, minimum width=0.1 cm]{};
\node(b11) [right=2.5cm of b8,mystyle,pattern=horizontal lines, pattern color=red, minimum width=0.15 cm]{};
\node(b12) [right=2.5cm of b9,mystyle,pattern=north east lines, pattern color=red, minimum width=0.2 cm]{};

\node(h5) [right=0.75cm of h4, minimum width=\pi cm]{$\ldots$};
\end{tikzpicture}
\caption{A block $B$ of jobs in $\sigma$.
The blue tasks are $j_1,j_2$ and $j_3$, while the red tasks are $j'_1,j'_2$ and $j'_3$ (in this order).}\label{fig:nphardcmax}
\end{figure}
\end{proof}

\begin{theorem}
$1\|(p_j,L,p_j)\|\Cs$ is strongly $\NP$-hard.
\end{theorem}

\begin{proof}
We reduce the strongly $\NP$-hard problem $1\|(p_j,L,p_j)\|C_{\max}$, proven to be strongly $\NP$-hard in \autoref{thm:NPpjLpjCmax}, to $1\|(p_j,L,p_j)\|\Cs$.
Consider an instance $I$ of $1\|(p_j,L,p_j)\|C_{\max}$. We define an instance $I'$ of $1\|(p_j,L,p_j)\|\Cs$ as follows.  
There are $n+M$ jobs in $I'$, where $M$ is a sufficiently large number.
\begin{itemize}
\item the first $n$ jobs are the same as the jobs in $I$ (small jobs),
\item for the remaining $M$ jobs, we have $p_j=\sum_{i=1}^n p_i+nL$, $j=n+1,\ldots,n+M$ (large jobs).
\end{itemize}

%$C$ be the optimal makespan for instance $I$ and

We prove that there is a solution $\sigma$ for $I$ with makespan at most $C$ if and only if there is a solution for $I'$ with a  total completion time of at most $z:=(n+M)C+h M(M+1)/2$, where $h:= 2\left(\sum_{i=1}^n p_i+nL\right)+L=a_j+L+b_j$ is the time required for scheduling a large job $j$.
The Theorem follows from this statement.

If there is such an $\sigma$, then we define $\sigma'$ (a solution of $I'$) from $\sigma$  as follows.
The small jobs start exactly at the same time as they start in $\sigma$, while the large jobs start right after them as soon as possible (in arbitrary order), see \autoref{fig:nphard}.
Observe that the completion time of any small job in $\sigma'$ is at most $C$, while the total completion time of the large jobs is $MC+h M(M+1)/2$.
Therefore, the total completion time of $\sigma'$ is at most $z$.

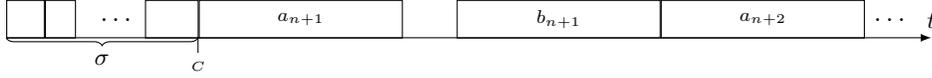
\begin{figure}
\centering
\begin{tikzpicture}
\def\ox{0} 
\def\oy{0} 
\coordinate(o) at (\ox,\oy);

%axis
\def\tl{12.3}
\draw [-latex](\ox,\oy) node[above left]{} -- (\ox+\tl,\oy) node[above,font=\small]{$t$};

%definitions for jobs
\def\pi{0.5}
\tikzstyle{mystyle}=[draw, minimum height=0.5cm,rectangle, inner sep=0pt,font=\scriptsize]
\tikzstyle{mystyle}=[draw, minimum height=0.5cm,rectangle, inner sep=0pt,font=\scriptsize]

\draw (2.55,0) -- (2.55,-0.2) node[below] {\tiny $C$};
\draw [decorate,decoration={brace,amplitude=3pt},xshift=0pt,yshift=0pt]
(2.55cm,0cm) -- (0cm,0cm) node [below=0.1cm,black,midway]{\small  $\sigma$};

%jobs
\node(b1) [above right=-0.01cm and 0cm of o,mystyle, minimum width=\pi cm]{};
\node(b2) [right=0cm of b1,mystyle, minimum width=0.4 cm]{};
\node(d2) [right=0.2cm of b2, minimum width=0.3 cm]{$\dots$};
\node(b6) [right=0cm of d2,mystyle, minimum width=0.7 cm]{  };

\node(h1) [right=0cm of b6,mystyle, minimum width=2.7 cm]{$a_{n+1}$};
\node(d2) [right=0.2cm of h1, minimum width=\pi cm]{};
\node(h2) [right=0cm of d2,mystyle, minimum width=2.7 cm]{$b_{n+1}$};
\node(h3) [right=0cm of h2,mystyle, minimum width=2.7 cm]{$a_{n+2}$};
\node(h4) [right=0cm of h3, minimum width=\pi cm]{$\ldots$};
\end{tikzpicture}
\caption{Schedule $\sigma'$ created from schedule $\sigma$}\label{fig:nphard}
\end{figure}

Now suppose that there is no such solution for $I$, i.e., the makespan $\Cmax$ of any solution $\sigma$ is at least $C+1$.
Consider an arbitrary optimal solution $\hat{\sigma}$ for $I'$.
Observe that no job task can be scheduled between the first and the second task of any large job in $\hat{\sigma}$, since the processing times of both tasks of any large job are larger than~$L$.
Furthermore,  no large job can precede any of the small jobs, otherwise, we would get a better schedule by moving that large job to the end of the schedule (cf., the definition of the processing times of the large jobs).  
Therefore, the large jobs start right after the small jobs, i.e., after $\Cmax$.
Hence, 
\begin{align*}
\sum_{j=1}^{n+M} \hat{C}_j &= \sum_{j=1}^n \hat{C}_j + MC_{\max} + h M(M+1)/2 \\
&>   M(C+1) + h M(M+1)/2>z,
\end{align*}
where the last inequality follows if $M$ is larger than $nC$.
\end{proof}

%\begin{remark}
%$1|(a, L_j, b, b \leq a)|\Cs$ is a generalization of $1|(p, L_j, p)|\Cs$, and $1|(p, L_j, p)|\Cs$ is known to be $\NP$-hard~\cite{chen2021}, this directly implies the $\NP$-hardness of $1|(a, L_j, b, b \leq a)|\Cs$.
%\end{remark}

\begin{theorem}
$1\|(1,L_j,1,\pi_a)\|\Cs$ is strongly $\NP$-hard.
\end{theorem}

\begin{proof}
We reduce the known strongly $\NP$-hard problem $1\|(1, L_j, 1, \pi_a)\|\Cmax$, proven to be strongly $\NP$-hard by Condotta and Shakhlevich~\cite{condotta12}, to $1\|(1, L_j, 1, \pi_a)\|\Cs$. Recall that $\pi_a$ fixes a scheduling sequence for the first tasks of all jobs. 
%As $1\|(1, L_j, 1)\|\Cs$ is a generalization of $1\|(1, L_j, 1, \pi_a)\|\Cs$, this result directly implies the $\NP$-hardness of $1\|(1, L_j, 1)\|\Cs$.
Consider an instance $I$ of $1\|(1, L_j, 1, \pi_a)\|\Cmax$.
We want to know if $I$ has a solution with $\Cmax \leq C$.
We define an instance $I'$ of $1\|(1, L_j, 1, \pi_a)\|\Cs$ as follows.  
There are $n+M$ jobs in $I'$, where $M$ is a sufficiently large number.
\begin{itemize}
\item the first $n$ jobs are the same as the jobs in $I$ (original jobs),
\item for the remaining $M$ jobs, we have $L_j = C + 2(j-n-1)$, $j=n+1,\ldots,n+M$ (helper jobs).
\end{itemize}
The fixed sequence of the first tasks in $I'$ is defined as $\pi'_a$, with
\begin{itemize}
\item $\pi'_a := (n+M, n+M-1, \dots, n+1, \{\pi_a\})$.
\end{itemize}

We prove that there is a solution $\sigma$ for $I$ with makespan at most $C$ if and only if there is a solution for $I'$ with a  total completion time of at most 
\begin{align*}
z:=n (M  + C) + M^2 + \frac{M(M+1)}{2} + MC.
\end{align*}

If there is such an $\sigma$, then we define $\sigma'$ (a solution of $I'$) from $\sigma$ as follows.
The helper jobs are scheduled in decreasing order of their indices one after another as soon as possible, while the original jobs are scheduled in the exact same way as in $\sigma$, but with their starting time increased by $M$ each. 
See the first schedule of \autoref{fig:nphard_1Lj1} for an illustration.
This schedule respects the order given by $\pi'_a$.
Observe that the completion time of any original job in $\sigma'$ is at most $M + C$, while the total completion time of the helper jobs is exactly $M^2 + \frac{M(M+1)}{2} + MC$.
Therefore, the total completion time of $\sigma'$ is at most $z$.

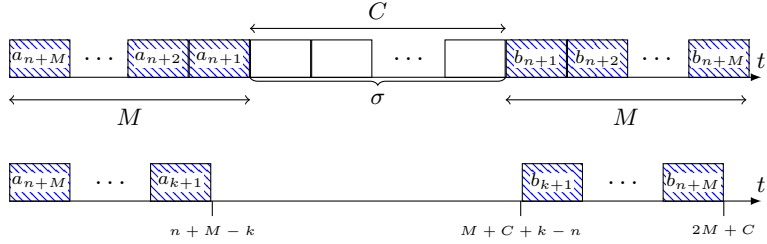
\begin{figure}
\centering
\begin{tikzpicture}
\def\ox{0} 
\def\oy{0} 
\coordinate(o) at (\ox,\oy);

%axis
\def\tl{10.0}
\draw [-latex](\ox,\oy) node[above left]{} -- (\ox+\tl,\oy) node[above,font=\small]{$t$};

\def\ll{0.5}

%definitions for jobs
\def\pi{0.5}
\tikzstyle{mystyle}=[draw, minimum height=0.5cm,rectangle, inner sep=0pt,font=\scriptsize,minimum width=0.8 cm]
\tikzstyle{mystyle2}=[draw = none, minimum height=0.25cm,rectangle, inner sep=0pt,font=\scriptsize]

%\draw (0,0) -- (0,-0.2) node[below] {$S_1$};
%\draw (\pi,0) -- (\pi,-0.2) node[below] {$S_2$};
%\draw (2.5,0) -- (2.5,-0.2) node[below] {$S_3$};
%\draw (7.5,0) -- (7.5,-0.2) node[below] {$C_1$};
%\draw (4.75,0) -- (4.75,-0.2) node[below] {$C_2$};
%\draw (9.5,0) -- (9.5,-0.2) node[below] {$C_3$};

%jobs
\node(h0) [above right=-0.01cm and 0cm of o,mystyle,pattern=north west lines, pattern color=blue]{};
\node(h0_t) [mystyle2, fill = white] at (h0.center) {$a_{n+M}$};
\node(b1) [above right=-0.01cm and 3.2cm of o,mystyle]{};
\node(b2) [right=0cm of b1,mystyle]{};
%\node(b3) [right=0cm of b1,mystyle]{};
\node(b3) [right=0.15cm of b2]{$\ldots$};
\node(b4) [right=0.15cm of b3,mystyle]{};
%\node(b5) [right=0.5cm of b3,mystyle]{};
%\node(b6) [right=0cm of b4,mystyle]{};
\node(h1) [left=0cm of b1,mystyle,pattern=north west lines, pattern color=blue]{};
\node(h1_t) [mystyle2, fill = white] at (h1.center) {$a_{n+1}$};
\node(h3) [left=0 cm of h1,mystyle,pattern=north west lines, pattern color=blue]{};
\node(h3_t) [mystyle2, fill = white] at (h3.center) {$a_{n+2}$};
%\node(h5) [left=0 cm of h3,mystyle,pattern=north west lines, pattern color=blue]{};
%\node(h5_t) [mystyle2, fill = white] at (h5.center) {$a_{n+3}$};
\node(h2) [right=0cm of b4,mystyle,pattern=north west lines, pattern color=blue]{};
\node(h2_t) [mystyle2, fill = white] at (h2.center) {$b_{n+1}$};
\node(h4) [right=0 cm of h2,mystyle,pattern=north west lines, pattern color=blue]{};
\node(h4_t) [mystyle2, fill = white] at (h4.center) {$b_{n+2}$};
%\node(h6) [right=0 cm of h4,mystyle,pattern=north west lines, pattern color=blue]{};
%\node(h6_t) [mystyle2, fill = white] at (h6.center) {$b_{n+3}$};
\node(h7) [right=0.8 cm of h4,mystyle,pattern=north west lines, pattern color=blue]{};
\node(h7_t) [mystyle2, fill = white] at (h7.center) {$b_{n+M}$};
\node(d1) [right=0.05cm of h0]{$\dots$};
\node(d2) [right=0.05cm of h4]{$\dots$};

\draw [<->] (3.2,0.65)--node[above]{\small $C$}(6.6,0.65); 
\draw [<->] (0,-0.3)--node[below]{\small $M$}(3.2,-0.3);
\draw [<->] (9.8,-0.3)--node[below]{\small $M$}(6.6,-0.3);

\draw [decorate,decoration={brace,amplitude=3pt},xshift=0pt,yshift=0pt]
(6.6,0) -- (3.2,0) node [below=0.1cm,black,midway]{\small $\sigma$};
\end{tikzpicture}
%\caption{Schedule $\sigma'$ created from schedule $\sigma$. 
%The original jobs are white, while the helper jobs are blue.}\label{fig:nphard_1Lj1}

\vspace{10pt}
\begin{tikzpicture}
\def\ox{0} 
\def\oy{0} 
\coordinate(o) at (\ox,\oy);

%axis
\def\tl{10.0}
\draw [-latex](\ox,\oy) node[above left]{} -- (\ox+\tl,\oy) node[above,font=\small]{$t$};

\def\ll{0.5}

%definitions for jobs
\def\pi{0.5}
\tikzstyle{mystyle}=[draw, minimum height=0.5cm,rectangle, inner sep=0pt,font=\scriptsize,minimum width=0.8 cm]
\tikzstyle{mystyle2}=[draw = none, minimum height=0.25cm,rectangle, inner sep=0pt,font=\scriptsize]

%\draw (0,0) -- (0,-0.2) node[below] {$S_1$};
\draw (2.7,0) -- (2.7,-0.2) node[below] {\tiny $n+M-k$};
\draw (6.8,0) -- (6.8,-0.2) node[below] {\tiny $M+C+k-n$};
\draw (9.5,0) -- (9.5,-0.2) node[below] {\tiny $2M+C$};
%\draw (4.75,0) -- (4.75,-0.2) node[below] {$C_2$};
%\draw (9.5,0) -- (9.5,-0.2) node[below] {$C_3$};

%jobs
\node(h0) [above right=-0.01cm and 0cm of o,mystyle,pattern=north west lines, pattern color=blue]{};
\node(h0_t) [mystyle2, fill = white] at (h0.center) {$a_{n+M}$};
\node(b3) [right=0.2cm of h0]{$\ldots$};
\node(h1) [right=0.2cm of b3,mystyle,pattern=north west lines, pattern color=blue]{};
\node(h1_t) [mystyle2, fill = white] at (h1.center) {$a_{k+1}$};
\node(h2) [right=2.8cm of b1,mystyle,pattern=north west lines, pattern color=blue]{};
\node(h2_t) [mystyle2, fill = white] at (h2.center) {$b_{k+1}$};
\node(h4) [right=0.2 cm of h2]{$\ldots$};
\node(h6) [right=0.2 cm of h4,mystyle,pattern=north west lines, pattern color=blue]{};
\node(h6_t) [mystyle2, fill = white] at (h6.center) {$b_{n+M}$};

%\draw [<->] (3,0.65)--node[above]{$C$}(6.4,0.65); 
%\draw [<->] (0,-0.3)--node[below]{$M$}(3,-0.3);
%\draw [<->] (9.4,-0.3)--node[below]{$M$}(6.4,-0.3);

%\draw [decorate,decoration={brace,amplitude=3pt},xshift=0pt,yshift=0pt]
%(6.4,0) -- (3,0) node [below=0.1cm,black,midway]{$S$};
\end{tikzpicture}
\caption{Above: schedule $\sigma'$ created from schedule $\sigma$. 
The original jobs are white, the helper jobs are blue.
Below: jobs $k+1,\ldots,n+M$ in schedule $\hat{\sigma}$.}\label{fig:nphard_1Lj1}

\end{figure}

Now suppose that there is no such solution for $I$, i.e., the makespan $\Cmax$ of any solution $\sigma$ is at least $C+1$.
Consider an arbitrary optimal solution $\hat{\sigma}$ for $I'$.
Suppose for the sake of contradiction that the total completion time of $\hat{\sigma}$ is at most $z$ for any~$M$.

Due to the fixed order of the first tasks $\pi'_a$ we can  assume that $a_{n+M}$ starts at $0$ in $\hat{\sigma}$.
Observe that the total completion time of the helper jobs is at least $M^2 + \frac{M(M+1)}{2}+MC$.
The original jobs start after the first tasks of the helper jobs, i.e., after $M$, thus their total completion time is at least $nM$. 
If any of the original jobs completes after $b_{n+M}$ in $\hat{\sigma}$, then its completion time is larger than $2M$, and the total completion time of the jobs in $\hat{\sigma}$ is larger than $M^2 + \frac{M(M+1)}{2}+MC+(n+1)M>z$, if $M>nC$. 
Thus, each original job has to start before $b_{n+M}$.  

If there is no gap among the first tasks of the helper jobs in $\hat{\sigma}$, then the machine is busy with the helper jobs in $[0,M]$ and in $[M+C,2M+C]$. 
Since each original job has to be completed before $b_{n+M}$ (i.e., before $2M+C$), these jobs have to be scheduled in $[M,M+C]$.
However, this is a contradiction to the the makespan of $\sigma$.

Therefore, in the following we suppose there is a gap between some first tasks of the helper jobs in $\hat{\sigma}$.
Let $k$, $n+1 \leq k \leq n+M-1$, be the largest index such that there is a gap before $a_k$.
Then, the machine is busy with jobs $n+M, n+M-1,\ldots, k+1$ in $[0,n+M-k]$ and in $[M+C+k-n,2M+C]$, see the second schedule in \autoref{fig:nphard_1Lj1}.
Since $a_k$ starts later than $n+M-k$, $b_k$ starts later than $M+C+k-n$, which means it starts later than $2M+C$ to avoid intersection.

Let $j$ be a job in $\{n+2, \dots, k\}$ whose second task $b_j$ is scheduled after $b_{n+M}$.
Since $L_{j-1}= L_j-2$, and $a_{j-1}$ has to be scheduled after $a_j$ according to $\pi'_a$, $b_{j-1}$ is scheduled either immediately before, or at some time after $b_j$.
In both cases, $b_{j-1}$ has to be scheduled after $b_{n+M}$ to avoid intersection with $b_{n+M}$.
This particularly implies that $b_{n+1}$ is scheduled after $b_{n+M}$, i.e., $C_{n+1}\geq C_{n+M}+1=2M+C+1$.

%Since $L_{k-1}= L_k-2$ and $a_{k-1}$ starts after $a_k$, $b_{k-1}$ starts either immediately before, or somewhere after $b_k$. 
%In both cases $b_{k-1}$ starts after $b_{n+M}$.
%An analogous proof provides that the same holds for each  job $j$, where $n+1\leq j\leq k$.  
%
%It also holds that, for any helper job $k$, $n+1 \leq k \leq j$, if $b_k$ is scheduled after $b_{n+M}$, also $b_{y-1}$ is scheduled after $b_{n+M}$.
%This stems from the fact that $L_{y-1} = L_y - 2$, and $a_{y-1}$ has to be scheduled after $a_y$ according to $\pi'_a$.
%This then directly implies that, if there are any gap between the first tasks of any helper jobs, $b_{n+1}$ has to be scheduled after $b_{n+M}$.
%Thus, the completion time of job $n+1$ is at least $2M + C + 1$, increased by $M$ in comparison to a scheduling of the helper jobs with no gaps.
The total completion time of jobs $n+2,\ldots,n+M$ is then at least $(M+C+2)+(M+C+3)+\ldots+(2M+C)=(M-1)(M+C+2)+(M-1)(M-2)/2$.
The total completion time of the original jobs is at least $nM$. Thus 
%$M^2 +\frac{M(M-1)}{2} + MC + M$ and the total completion time of the original jobs is at least $nM$, 
we have the following lower bound on the total completion time of $\hat{\sigma}$: 
\begin{align*}
&(2M+C+1)+\left((M-1)(M+C+2)+\frac{(M-1)(M-2)}{2}\right)+(nM)\\
&=M^2+M+CM+\frac{M^2+M}{2}+nM>z,
\end{align*}
if $M > nC$, proving the Theorem.
\end{proof}

\section{Approximation results}
\label{sec:main}

As in the following, we only look at CTP with the objective of minimizing the total completion times, we call this problem only "CTP" from now on for simplicity reasons.
In this section we give polynomial-time approximation algorithms for a number of CTP variants.
All of these variants are proven to be $\NP$-hard either by Chen and Zhang~\cite{chen2021}, by Kubiak~\cite{kubiak22}, or by the results of \autoref{sec:complex}. 
We start this section with two useful lemmas that provide lower bounds on the objective value of any optimal solution for the general CTP.
\autoref{subsec:fixp} and \autoref{subsec:fixL} consider variants with fixed processing times and fixed delay times, respectively.
\autoref{subsec:relatedpL} examines variants where there exists some relation between the processing times and the delay time of each job.

\subsection{Lower bounds on the optimum}
\label{sec:lbopt}

Recall the definitions $\Coptf_j$ and $\Copts_j$ as the completion time of the $j^{\textnormal{th}}$ finishing and $j^{\textnormal{th}}$ starting job of some optimal schedule, respectively.
%is the completion time of the $j^{\textnormal{th}}$ finishing job in a fixed optimal schedule, i.e., there are $j$ jobs that completes until  $\Coptf_j$ in that schedule. 
The next lemma is straightforward from the definition of $\Coptf_j$, since there are $j$ jobs that complete until $\Coptf_j$ in an optimal schedule.

\begin{lemma}
\label{lemma:lbfinish}
Let the jobs of a CTP instance be indexed in non-decreasing $a_i + b_i$ order. 
Then, for any optimal schedule for this instance, we have 
\begin{enumerate}
\item $\Coptf_{j} \geq \sum_{i=1}^{j} (a_{i}+b_{i})$, $j=1,\ldots,n$ and
\item $\Copt \geq \sum_{j=1}^n \sum_{i=1}^{j} (a_i+b_i)$.
\end{enumerate}
\end{lemma}

%\begin{proof}
%Consider an optimal schedule. 
%Due to the non-decreasing $a_j+b_j$ order, we have $\Copt_{j} \geq \sum_{i=1}^{j} (a_{i}+b_{i})$ and $\Copt=\sum_{j=1}^n \Coptf_j \geq \sum_{j=1}^{n} \sum_{i=1}^{j} (a_{i}+b_{i})$.
%\end{proof}

The second lemma is analogous, and follows from the observation that there are $j-1$ first tasks that finish until the first task of the job corresponding to $\Copts_j$ starts in some optimal schedule.
See \autoref{fig:boundinglemma} for an illustration.

\begin{lemma}
\label{lemma:lbstart}
Let the jobs of a CTP instance be indexed in non-decreasing $a_i$ order.
% and $\sigma^{OPT}$ be an arbitrary optimal schedule. Let $\sigma^{OPT}(j)$ be the starting time and $\Copts_j$ the completion time of the $j^{\textnormal{th}}$ starting job in $\sigma^{OPT}$.
Then, for any optimal schedule for this instance, where $j'$ is the $j^{\textnormal{th}}$ starting job, we have
\begin{enumerate}
\item $\Copts_j\geq \sum_{i=1}^{j} a_i + L_{j'} + b_{j'}$, $j=1,\ldots,n$ and
\item $\Copt \geq \sum_{j=1}^n \sum_{i=1}^{j} a_i + \sum_{j=1}^n L_j + \sum_{j=1}^n b_j$.
\end{enumerate}
\end{lemma}

\begin{figure}
\centering

\begin{tikzpicture}

\def\ox{0} 
\def\oy{0} 
\coordinate(o) at (\ox,\oy);

%axis
\def\tl{9.0}
\draw [-latex](\ox,\oy) node[above left]{} -- (\ox+\tl,\oy) node[above,font=\small]{$t$};

%\draw<1>[<->] (0.5,0.6)--node[above=0]{$L$}(3,0.6); 

%definitions for jobs
\def\pi{0.5}
\tikzstyle{mystyle}=[draw, minimum height=0.5cm,rectangle, inner sep=0pt,font=\scriptsize]
\tikzstyle{mystyle}=[draw, minimum height=0.5cm,rectangle, inner sep=0pt,font=\scriptsize]

%jobs
\node(b1) [above right=-0.01cm and -0.01cm of o,mystyle, minimum width=0.5 cm]{};
\node(b2) [right=0cm of b1,mystyle, minimum width=0.5 cm]{};
\node(b2) [right=0.5cm of b2,mystyle, minimum width=0.5 cm]{};
\node(b3) [right=0.8cm of b2]{$\ldots$};
%\node(b4) [right=2.5cm of b3,mystyle, minimum width=0.5 cm,pattern=north west lines, pattern color=green]{$b_2$};
\node(b5) [right=2.15cm of b2,mystyle, minimum width=0.8 cm]{$a_{j'}$};
\node(b6) [right=2.5cm of b5, mystyle, minimum width=0.5 cm]{$b_{j'}$};

\draw [<->] (0,0.65)--node[above]{\small $\geq\sum_{i=1}^j a_i$}(5,0.65);
\draw [<->] (5,0.65)--node[above]{\small $L_{j'}$}(7.5,0.65);
\draw (8,0) -- (8,-0.2) node[below] {\tiny $\Copts_j$};
\end{tikzpicture}
\caption{Illustration of the bound on $\Copts_j$.}\label{fig:boundinglemma}
\end{figure}
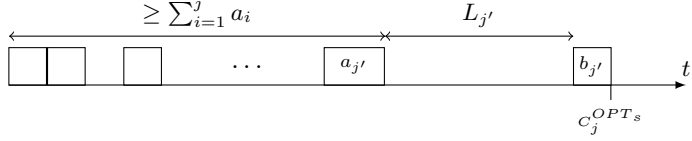

\subsection{CTP with fixed processing times}\label{subsec:fixp}
In this section we assume $a_j=a$ and $b_j=b$ for each job $j\in\J$.
Consider \autoref{algo:A}.

\begin{algorithm}[htp!]
\caption{}
\label{algo:A}
\textbf{Input}:  a CTP instance with $a_j = a$ and $b_j = b$ for each job $j$

\textbf{Output}: a schedule $\sigma$ for this instance

\begin{enumerate}
    \item Sort the jobs in non-decreasing $L_j$ order.  In the following, let the jobs be indexed in this order.\label{step:A_sort_Lj}
    \item Schedule the jobs one-by-one as soon as possible without intersection.\label{step:A_scheduled}
\end{enumerate}
\end{algorithm}

Recall that here, $\Salg$ denotes the schedule created \autoref{algo:A} and $S_j$ is the starting time of job $j$.
The precise meaning of 'as soon as possible' is the following: for a given (partial) schedule $\sigma^p$ and a job $j$, scheduling~$j$ as soon as possible means setting it's starting time $S_j:=\min\{t\geq 0: \text{ the machine is idle both in } [t,t+a_j] \text{ and in } [t+a_j+L_j, t+a_j+L_j+b_j]\}$.

%Let $S$ denote the schedule created by \autoref{algo:A}.
%The following lemma can be proved by simple induction.

%The next lemma describes an important observation that we need for proving the subsequent Theorem.

\begin{lemma}\label{obs:algA}
Suppose that $a_j=a$, $b_j=b$ ($j\in\J$) and $b\leq a$.
When \autoref{algo:A} schedules job $j$, where $j$ is the $j^{\textnormal{th}}$ job in non-decreasing $L_j$ order, then $a_j$ starts from the earliest time point $t'$, such that the machine is idle in $[t',t'+a]$. 
\end{lemma}
\begin{proof}
We prove the statement by an induction on $j$ in the order given by step~\ref{step:A_sort_Lj} in \autoref{algo:A}; it is trivial for $j=1$. 
Suppose that the statement is true for some $j$, and consider the partial schedule $\sigma^p$ created by the algorithm before scheduling job $j+1$.
Let $t$ be the earliest time point such that the machine is either idle in $[t',t'+a]$ in $\sigma^p$ (it is possible that  $t'$ is the time point when $\sigma^p$ finishes).
We prove the lemma by showing $S_{j+1} = t$.
Due to the definition of $t'$, $a_{j+1}$ cannot intersect with any task in $\sigma^p$. We need to prove the same for $b_{j+1}$, which is scheduled in $[t'+a+L_{j+1},t'+a+L_{j+1}+b]$.

First, observe that $a_i$ starts before $a_{i+1}$ for each $i\leq j$ due to the induction.
This means $a_1,\ldots,a_j$ complete before $t'$.
Let $L_1, \dots, L_n$ be the delay times of the jobs obtained in step~\ref{step:A_sort_Lj} of \autoref{algo:A}.
Since $L_1,\ldots,L_j\leq L_{j+1}$, we have $C_1,\ldots,C_j\leq t'+L_{j+1}+b\leq t'+L_{j+1}+a$, where the last inequality follows from $b\leq a$.
Thus, as each task in $\sigma^p$ completes before the start of $b_{j+1}$, there is no intersection of any scheduled tasks. 
\end{proof}

\begin{lemma}\label{lemma:algAtime}
\autoref{algo:A} runs in $\bigO(n \log n)$ time and it always produces a feasible solution.
\end{lemma}
\begin{proof}
Sorting the jobs requires $\bigO(n\log n)$ time.
When \autoref{algo:A} schedules job $j$, it searches the first gap after the first task of the directly previous scheduled job with a length of at least $a$ (\autoref{obs:algA}).
This means the length of the gap after each task is only checked once during the whole procedure, which requires $\bigO(n)$ time in total. 
The feasibility of the schedule is straightforward from the definition of 'as soon as possible'. 
\end{proof}

\begin{theorem}
\label{lemma:aLjbsmallerb}
\autoref{algo:A} is a factor-$2$ approximation for $1\|(a, L_j, b, b \leq a)\|\Cs$.
\end{theorem}

\begin{proof}
Due to \autoref{lemma:algAtime}, it remains to prove that $\Calg\leq 2\Copt$.

From \autoref{obs:algA} we know that the machine is always busy just before a first task of a job is scheduled (except $a_1$, which starts at time $0$).
This means there are at most $j-1$ gaps before $S_j$, and the length of any of these gaps is smaller than $a$.
Hence, if we consider the partial schedule at the time when the algorithm schedules job $j$, we have 
\begin{align*}
S_j\leq(j-1)(a+b)+(j-1)a,\quad j\in\J,
\end{align*}
because (i) the machine is busy at most $(j-1)(a+b)$ time in $[0,S_j]$ with tasks of the jobs $1,\ldots,j-1$; and  (ii) the total length of the gaps before $S_j$ is smaller than~$(j-1)a$.
Therefore, we have $C_j \leq (j-1) (a+b) + (j-1) a + a + L_j + b$ and

\begin{align*}
\Calg &=\sum_{j=1}^n C_j\leq \frac{n (n-1) }{2}(a+b) + \frac{n (n-1) }{2}a + n (a+b) + \sum_{j=1}^n L_j\\
&\leq \Copt + \frac{n (n+1) }{2}a + nb + \sum_{j=1}^n L_j \leq 2\Copt,
\end{align*}
where the second inequality follows from \autoref{lemma:lbfinish} and from $n(n-1)/2+n=n(n+1)/2$, while the third follows from \autoref{lemma:lbstart}.
\end{proof}

As $1\|(a, L_j, b, b \leq a)\|\Cs$ is a more general version of the variant $1\|(p, L_j, p)\|\Cs$, \autoref{coll:pLjp} directly follows. 

\begin{corollary}
\label{coll:pLjp}
\autoref{algo:A} is a factor-$2$ approximation for $1\|(p, L_j, p)\|\Cs$.
\end{corollary}

The following lemma describes some important attributes of $\Salg$ in the opposite case of $a\leq b$.
\begin{lemma}\label{lem:aLjb}
Suppose that $a_j=a$, $b_j=b$ ($j\in\J$) and $a\leq b$.
Consider the partial schedule $\sigma^p$ created by \autoref{algo:A}, right before it schedules job $j$ starting at $S_j$.
Then, 
\begin{enumerate}[(i)]
\item $a_j$ or $b_j$ start right after another task (for each $j>1$),\label{lem:aLjb_ra}
\item the number of gaps before $S_j$ is at most $j-1$,\label{lem:aLjb_no_gap}
\item the length of each gap in $[0,S_j]$ is at most $b$.\label{lem:aLjb_length_gap}
\end{enumerate}
\end{lemma}

\begin{proof}
Statement~(\ref{lem:aLjb_ra}) immediately follows from the algorithm.
We use induction on $j$ for proving the remaining statements (\ref{lem:aLjb_no_gap}) and (\ref{lem:aLjb_length_gap}); they are trivial for $j=1$.
Suppose that they are true for some $j\geq 1$, we then prove them for $j+1$.

Consider (\ref{lem:aLjb_no_gap}). If $a_{j+1}$ or $b_{j+1}$ either directly precede, or start right after some previously scheduled $a_k$, $k < j+1$, in $\sigma^p$, the statement immediately follows.
Otherwise, $a_{j+1}$ or $b_{j+1}$ either precede, or start right after some $b_k$ with $k<j+1$.
In the latter, $a_{j+1}$ either precedes, or starts right after $a_k$ due to the non-decreasing delay time order, from which the statement also immediately follows.
In the former (if $a_{j+1}$ starts right after $b_k$), (\ref{lem:aLjb_no_gap}) follows from the induction, since there is no gap in $\sigma^p$ right before $a_\ell$ or $b_\ell$ for each $\ell\leq j$ (as known from (\ref{lem:aLjb_ra})). Thus, there are at most $j$ tasks before $a_{j+1}$ which start right after a gap.

For proving (\ref{lem:aLjb_length_gap}), suppose for purpose of showing contradiction that there exists a gap $[g_1,g_2]$ before $b_k$ in $\sigma^p$ with a length larger than $b$.
We know from the induction that each $a_\ell$ ($\ell\leq j$) completes before $g_1$.
If $[g_1,g_2]$  existed however, the algorithm would schedule $a_{j+1}$ at time $g_1$: there would be neither an intersection between $a_{j+1}$ and any other task (since $a\leq b< g_2-g_1$), nor between $b_{j+1}$ and any other task (since $a_{j+1}$ starts later than the previously scheduled first tasks, and $L_{j+1}\geq L_{\ell}$ for each $\ell\leq j$). The existence of such a gap would contradict the definition of the algorithm, therefore, (\ref{lem:aLjb_length_gap}) follows.
\end{proof}

\begin{theorem}
\label{lemma:aLjbgreaterb}
\autoref{algo:A} is a factor-$3$ approximation for $1\|(a,L_j,b)\|\Cs$.
%with the condition that $a \leq b \leq \alpha a$, for $\alpha \geq 1$, with a factor of $1 + \frac{2 \alpha}{(1+\alpha)} \leq 3$.
\end{theorem}

\begin{proof}
Due to \autoref{lemma:algAtime}, it remains to prove that $\Calg\leq 3\Copt$.

If $b \leq a$, we can approximate $1\|(a,L_j,b)\|\Cs$ with a factor of $2$ (\autoref{lemma:aLjbsmallerb}). 
Therefore, w.l.o.g, we assume in the following that $a \leq b$. 
We show that, in this case, we get an approximation factor of $3$.

From \autoref{lem:aLjb} (\ref{lem:aLjb_no_gap}) and (\ref{lem:aLjb_length_gap}) we know that when the algorithm schedules job $j$, the total idle time of the machine in $[0,S_j]$ is at most $(j-1)b$. 
Due to the order of the jobs, the machine is busy in the partial schedule from $0$ to $S_j$ for at most  $(a+b) (j-1)$ time; thus we have 
$S_j\leq (a+b) (j-1) + b (j-1)$ and $C_j \leq (a+b) (j-1) + b (j-1) + (a +  L_j+b)$.
Hence,
\begin{align*}
\Calg &=\sum_{j=1}^n C_j\leq \sum_{j=1}^n (a+b) (j-1) + \sum_{j=1}^n b (j-1) + n (a + b) + \sum_{j=1}^n L_j\\
&= n(n-1)b+\left(\frac{n(n+1)}{2}a +nb+\sum_{j=1}^n L_j\right) \leq n(n-1)b+\Copt,\\
%&\leq \Copt + n(n-1)\alpha a,
\end{align*}
where the second inequality follows from \autoref{lemma:lbstart}.
Since $\Coptf_j\geq (a+b)j$ by \autoref{lemma:lbfinish}, we have $\Copt\geq \frac{n(n+1)}{2} (a + b)$. 
Therefore,
\begin{align*}
\Calg \leq \Copt + n(n-1) b \leq 3 \Copt \enspace . \tag*{\qedhere}
\end{align*}
\end{proof}

\begin{theorem}
\label{lemma:1Lj1}
\autoref{algo:A} is a factor-$1.5$ approximation for $1\|(1,L_j,1)\|\Cs$.
\end{theorem}

\begin{proof}
Due to \autoref{lemma:algAtime}, it remains to prove that $\Calg\leq 1.5\Copt$.
From \autoref{obs:algA} we know that the algorithm always schedules $a_j$ in the first gap.
Thus, the starting time of $j$ ($j\in\J$) is at most $2 (j-1)$, because the total processing time of the jobs  $\ell<j$ is $2 (j-1)$ and the other jobs start later. 
Therefore, $C_j\leq 2 (j-1) + 2 + L_j$, and 
\begin{align*}
\Calg &\leq 2 \sum_{j=1}^n (j-1) + 2 n + \sum_{j=1}^n L_j \\
&\leq  \Copt + \sum_{j=1}^n (j-1) + n \leq 1.5  \Copt,
\end{align*}
where the second inequality follows from \autoref{lemma:lbstart}, and the third from \autoref{lemma:lbfinish}.
\end{proof}

\subsection{CTP with fixed delay times}\label{subsec:fixL}

In this section we assume $L_j=L$ for each job $j\in\J$.
Consider \autoref{algo:B}, based on an idea of Ageev and Ivanov~\cite{ageev2016}.
Observe that both the first and the second tasks are in non-decreasing $a_j+b_j$ order in the schedule found by \autoref{algo:B}.
%Consider the following algorithm for solving CTP:

\begin{algorithm}
\caption{}
\label{algo:B}
\textbf{Input}: a CTP instance with $L_j = L$ for each job $j$

\textbf{Output}: a schedule $\sigma$ for this instance
\begin{enumerate}
    \item Sort the jobs in non-decreasing $a_j + b_j$ order. In the following, let the jobs be indexed in this order.
    \item Schedule $a_1$ from $0$ and $b_1$ from $a_1+L$. Let $s:=1$ and $j_s:=1$.
    \item For $j=2,\ldots,n$ do the following:
    \begin{enumerate}[i)]
        \item If it is possible to start $a_j$ immediately after $a_{j-1}$ without $b_j$ intersecting any job task on the schedule, then do it and consider the next job.\label{step:B_aj}
        \item If it is possible to start $b_j$ right after $b_{j-1}$ without $a_j$ intersecting any job tasks on the schedule, then do it and go to consider the next job.\label{step:B_bj}
        \item Otherwise, start $a_j$ immediately after $b_{j-1}$. Let $s:=s+1$ and $j_s:=j$.\label{step:B_scheduleend}
    \end{enumerate}

\end{enumerate}
\end{algorithm}

\begin{lemma}\label{lemma:algBtime}
\autoref{algo:B} runs in $\bigO(n \log n)$ time and always produces a feasible solution.
\end{lemma}

\begin{proof}
The run time of \autoref{algo:B} is straightforward.
Together with the observation that in step~\ref{step:B_scheduleend}, we schedule the current job after all the other tasks, it is clear that there are no intersections in the schedule produced by the algorithm.
\end{proof}

Let $B_s:=\{j_s, j_{s}+1,\ldots, j_{s+1}-1\}$ be the $s^{\textnormal{th}}$ {\emph{block}} ($1 \leq s \leq n$) and let $H_s:=C_{j_{s+1}-1}-S_{j_s}$ denote the length of $B_s$.
The next lemma describes an important observation on the gap sizes within a block and follows directly from the equal delay times and some simple algebraic calculations (see~\autoref{fig:blockgapsize} for illustration).

\begin{lemma}\label{lem:ajLbj}
If jobs $j$ and $j+1$ are in the same block, then $a_{j+1}$ starts immediately after $a_j$  or $b_{j+1}$ starts immediately after $b_{j}$.
In the former case the gap between the second tasks is at most $a_{j+1}-b_j$, while in the latter case the gap between the first tasks is at most $b_j-a_{j+1}$.
\end{lemma}

\begin{figure}
\centering

\begin{tikzpicture}

\def\ox{0} 
\def\oy{0} 
\coordinate(o) at (\ox,\oy);

%axis
\def\tl{11.0}
\draw [-latex](\ox,\oy) node[above left]{} -- (\ox+\tl,\oy) node[above,font=\small]{$t$};

%\draw<1>[<->] (0.5,0.6)--node[above=0]{$L$}(3,0.6); 

%definitions for jobs
\def\pi{0.5}
\tikzstyle{mystyle}=[draw, minimum height=0.5cm,rectangle, inner sep=0pt,font=\scriptsize]
\tikzstyle{mystyle2}=[draw = none, minimum height=0.25cm,rectangle, inner sep=0pt,font=\scriptsize]

%jobs
\node(b1) [above right=-0.01cm and -0.01cm of o,mystyle, minimum width=0.5 cm,pattern=north west lines, pattern color=red]{};
\node(b1_t) [mystyle2, fill = white] at (b1.center) {$a_1$};
\node(b2) [right=4cm of b1,mystyle, minimum width=0.5 cm,pattern=north west lines, pattern color=red]{};
\node(b2_t) [mystyle2, fill = white] at (b2.center) {$b_1$};
\node(b3) [right=0cm of b1,mystyle, minimum width=1 cm,pattern=north east lines, pattern color=green]{};
\node(b3_t) [mystyle2, fill = white] at (b3.center) {$a_2$};
\node(b4) [right=4cm of b3,mystyle, minimum width=1.3 cm,pattern=north east lines, pattern color=green]{};
\node(b4_t) [mystyle2, fill = white] at (b4.center) {$b_2$};
\node(b5) [right=0.5cm of b3,mystyle, minimum width=0.8 cm,pattern=vertical lines, pattern color=yellow]{};
\node(b5_t) [mystyle2, fill = white] at (b5.center) {$a_3$};
\node(b6) [right=4cm of b5, mystyle, minimum width=2 cm,pattern=vertical lines, pattern color=yellow]{};
\node(b6_t) [mystyle2, fill = white] at (b6.center) {$b_3$};

%arrows
\draw [<->] (0.5,0.65)--node[above]{\small $L$}(4.5,0.65);
\draw [<->] (1.5,-0.3)--node[below]{\small $(b_2 - a_3)$}(2,-0.3);
\draw [<->] (5,-0.3)--node[below]{\small $(a_2 - b_1)$}(5.5,-0.3);
\end{tikzpicture}

\caption{Example for gaps in different cases.
%$a_2$ starts immediately after $a_1$, thus the gap between $b_1$ and $b_2$ is $a_2-b_1$; $b_3$ starts immediately after $b_2$, thus the gap between $a_2$ and $a_3$ is $b_2-a_3$.
}
\label{fig:blockgapsize}
\end{figure}
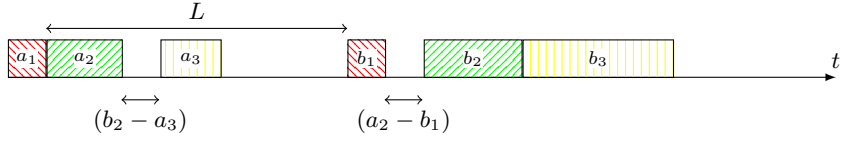

\begin{theorem}
\label{lemma:ajLbj}
\autoref{algo:B} is a factor-$3$ approximation for $1\|(a_j,L,b_j)\|\Cs$.
\end{theorem}

\begin{proof}
Due to \autoref{lemma:algBtime}, it remains to prove that $\Calg\leq 3\Copt$.
The length of a block $B_s$ is the sum of the following: (i) the lengths of the first tasks of the jobs in $B_s$,  (ii) the lengths of the gaps among these tasks, (iii) $L$, and (iv) the length of the second task of the last job in $B_s$, i.e., $H_s=\sum_{i=j_s}^{j_{s+1}-1}a_i+ \sum_{i=j_s}^{j_{s+1}-2}G_i+L+b_{j_{s+1}-1}$, where $G_i\geq 0$ is the length of the gap between $a_i$ and $a_{i+1}$.
From \autoref{lem:ajLbj}, we have $G_i\leq \|b_i-a_{i+1}\|\leq b_i$, thus
\begin{align}
H_s\leq \sum_{i=j_s}^{j_{s+1}-1}(a_i+b_i)+L. \label{eq:Hs_general}
\end{align}

Observe that the algorithm starts a new block every time it cannot schedule the next upcoming job in steps~\ref{step:B_aj}) and \ref{step:B_bj}).
Therefore, there can be two reasons why $j_{s+1}$ cannot be scheduled in $B_s$: (a) the gap between $a_{j_{s+1}-1}$ and $b_{j_s}$ is smaller than $a_{j_{s+1}}$ or (b) the completion time of $b_{j_{s+1}-1}$ minus the starting time of $b_{j_s}$ is larger than $L$, see~\autoref{fig:newblockcases}.

\begin{figure}
\centering

\begin{tikzpicture}

\def\ox{0} 
\def\oy{0} 
\coordinate(o) at (\ox,\oy);

%axis
\def\tl{11.2}
\draw [-latex](\ox,\oy) node[above left]{} -- (\ox+\tl,\oy) node[above,font=\small]{$t$};

%\draw<1>[<->] (0.5,0.6)--node[above=0]{$L$}(3,0.6); 

%definitions for jobs
\def\pi{0.5}
\tikzstyle{mystyle}=[draw, minimum height=0.5cm,rectangle, inner sep=0pt,font=\scriptsize]
\tikzstyle{mystyle2}=[draw = none, minimum height=0.25cm,rectangle, inner sep=0pt,font=\scriptsize]

%jobs
\node(b1) [above right=-0.01cm and -0.01cm of o,mystyle, minimum width=0.4 cm,pattern=north west lines, pattern color=red]{};
\node(b1_t) [mystyle2, fill = white] at (b1.center) {$a_1$};
\node(b2) [right=1cm of b1,mystyle, minimum width=0.4 cm,pattern=north west lines, pattern color=red]{};
\node(b2_t) [mystyle2, fill = white] at (b2.center) {$b_1$};
\node(b3) [right=0cm of b1,mystyle, minimum width=0.4 cm,pattern=north east lines, pattern color=green]{};
\node(b3_t) [mystyle2, fill = white] at (b3.center) {$a_2$};
\node(b4) [right=1cm of b3,mystyle, minimum width=0.5 cm,pattern=north east lines, pattern color=green]{};
\node(b4_t) [mystyle2, fill = white] at (b4.center) {$b_2$};
\node(b5) [right=0cm of b4,mystyle, minimum width=1.1 cm,pattern=vertical lines, pattern color=yellow]{};
\node(b5_t) [mystyle2, fill = white] at (b5.center) {$a_3$};
\node(b6) [right=1cm of b5, mystyle, minimum width=0.5 cm,pattern=vertical lines, pattern color=yellow]{};
\node(b6_t) [mystyle2, fill = white] at (b6.center) {$b_3$};
\node(b7) [right=0cm of b5,mystyle, minimum width=0.7 cm,pattern=horizontal lines, pattern color=cyan]{};
\node(b7_t) [mystyle2, fill = white] at (b7.center) {$a_4$};
\node(b8) [right=1.2cm of b7, mystyle, minimum width=1.8 cm,pattern=horizontal lines, pattern color=cyan]{};
\node(b8_t) [mystyle2, fill = white] at (b8.center) {$b_4$};
\node(b9) [right=0cm of b8,mystyle, minimum width=0.4 cm,pattern=crosshatch, pattern color=magenta]{};
\node(b9_t) [mystyle2, fill = white] at (b9.center) {$a_5$};
\node(b10) [right=1cm of b9, mystyle, minimum width=2.4 cm,pattern=crosshatch, pattern color=magenta]{};
\node(b10_t) [mystyle2, fill = white] at (b10.center) {$b_5$};

\draw [<->] (0.4,0.65)--node[above]{\small $L$}(1.4,0.65);
\draw [<->] (0.8,-0.3)--node[below]{\small $< a_3$}(1.4,-0.3);
\draw [<->] (4.5,-0.3)--node[below]{\small $>L$}(7.1,-0.3);
\end{tikzpicture}

\caption{Reasons of starting a new block: $a_3$ does not fit between $a_2$ and $b_1$, while  the total completion times of the second tasks in the second block ($b_3+b_4$) and the gap between them is larger than $L$.
%Example of a block schedule with blocks $B_1 = \{j_1, j_2\}$, $B_2 = \{j_3, j_4\}$ and $B_3 = \{j_5\}$. Job $j_3$ is scheduled in a new block because it's first task $a_3$ does not fit in the gap between $a_2$ and $b_1$, corresponding to case (a). Job $j_5$ is scheduled in a new block because the the completion time of $b_4$ minus the starting time of $b_3$ is larger than $L$, corresponding to case (b).
}
\label{fig:newblockcases}
\end{figure}
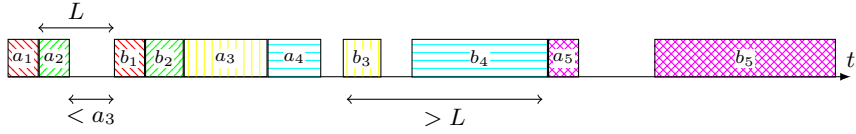

In case of (a), we have $\sum_{i=j_s+1}^{j_{s+1}}a_i+\sum_{i=j_s+1}^{j_{s+1}-1}b_i>L$, while in case of (b), we have  $\sum_{i=j_s+1}^{j_{s+1}-1}a_i+\sum_{i=j_s}^{j_{s+1}-1}b_i>L$ (from \autoref{lem:ajLbj}).
For each block, at least one of the previous inequalities holds, and neither of those inequalities contains any task occurring in any other inequality of another block. 
Hence, summing the valid inequalities for the first $s-1$ blocks, we have 
\begin{align}
(s-1)L<\sum_{i=1}^{j_{s}-1}(a_i+b_i)+a_{j_s}.\label{eq:boundL_general}    
\end{align}

Applying \autoref{eq:Hs_general} and \autoref{eq:boundL_general} to bound the completion time $C_j$ of a job $j$ in $B_s$, we get
\begin{align*}
C_j&\leq  \sum_{k=1}^{s-1} H_k+\sum_{i=j_s}^j (a_i+b_i)+L \leq \sum_{i=1}^j (a_i+b_i) +sL\\
&< 2\sum_{i=1}^{j} (a_i+b_i)+L\leq 3 \Coptf_j,     
\end{align*}
where the last inequality follows from \autoref{lemma:lbfinish} and from $L<\Coptf_j$.
Summing over all jobs, the Theorem follows.
\end{proof}

\begin{theorem}
\label{lemma:pjLpj}
\autoref{algo:B} is a factor-$1.5$ approximation for $1\|(p_j,L,p_j)\|\Cs$.
\end{theorem}

\begin{proof}
Due to \autoref{lemma:algBtime}, it remains to prove that $\Calg\leq 1.5\Copt$.

Consider an arbitrary optimal solution.
Note that, in the present case, the non-decreasing $a_j+b_j$ order is the same as the non-decreasing $a_j$ order. Thus, we can use both lower bounds on the optimum described in \autoref{sec:lbopt}. Also, $\Coptf_j=\Copts_j$ follows directly from all delay times being fixed.

Consider jobs $j$ and $j+1$ from the same block.
From \autoref{lem:ajLbj}, and from $p_j\leq p_{j+1}$, we know that there is no gap between $a_j$ and $a_{j+1}$ in $\Salg$. 
Thus, the length of a block $B_s$ can be expressed as
\begin{align}
H_s = \sum_{k=j_s}^{j_{s+1}-1} p_k + L + p_{j_{s+1}-1}. \label{eq:Hs}
\end{align}
Since $j_s$ could not be scheduled in $B_{s-1}$, we have 
\begin{align}
\sum_{k=j_{s-1}+1}^{j_s} p_k>L, \enspace s \geq 2, \label{eq:boundL}   
\end{align}
see~\autoref{fig:pjLpj}.

\begin{figure}
\centering

\begin{tikzpicture}

\def\ox{0} 
\def\oy{0} 
\coordinate(o) at (\ox,\oy);

%axis
\def\tl{11.0}
\draw [-latex](\ox,\oy) node[above left]{} -- (\ox+\tl,\oy) node[above,font=\small]{$t$};

%\draw<1>[<->] (0.5,0.6)--node[above=0]{$L$}(3,0.6); 

%definitions for jobs
\def\pi{0.5}
\tikzstyle{mystyle}=[draw, minimum height=0.5cm,rectangle, inner sep=0pt,font=\scriptsize]
\tikzstyle{mystyle2}=[draw = none, minimum height=0.25cm,rectangle, inner sep=0pt,font=\scriptsize]

%jobs
\node(b1) [above right=-0.01cm and -0.01cm of o,mystyle, minimum width=0.5 cm,pattern=north west lines, pattern color=red]{};
\node(b1_t) [mystyle2, fill = white] at (b1.center) {$p_1$};
\node(b2) [right=2.3cm of b1,mystyle, minimum width=0.5 cm,pattern=north west lines, pattern color=red]{};
\node(b2_t) [mystyle2, fill = white] at (b2.center) {$p_1$};
\node(b3) [right=0cm of b1,mystyle, minimum width=0.5 cm,pattern=north east lines, pattern color=green]{};
\node(b3_t) [mystyle2, fill = white] at (b3.center) {$p_2$};
\node(b4) [right=2.3cm of b3,mystyle, minimum width=0.5 cm,pattern=north east lines, pattern color=green]{};
\node(b4_t) [mystyle2, fill = white] at (b4.center) {$p_2$};
\node(b5) [right=0cm of b3,mystyle, minimum width=0.8 cm,pattern=vertical lines, pattern color=yellow]{};
\node(b5_t) [mystyle2, fill = white] at (b5.center) {$p_3$};
\node(b6) [right=2.3cm of b5, mystyle, minimum width=0.8 cm,pattern=vertical lines, pattern color=yellow]{};
\node(b6_t) [mystyle2, fill = white] at (b6.center) {$p_3$};
\node(b7) [right=0cm of b6,mystyle, minimum width=1.3 cm,pattern=horizontal lines, pattern color=cyan]{};
\node(b7_t) [mystyle2, fill = white] at (b7.center) {$p_4$};
\node(b8) [right=2.3cm of b7, mystyle, minimum width=1.3 cm,pattern=horizontal lines, pattern color=cyan]{$p_4$};
\node(b8_t) [mystyle2, fill = white] at (b8.center) {$p_4$};

%arrows
\draw [<->] (0.5,0.65)--node[above]{\small $L$}(2.8,0.65);
\draw [<->] (0.5,-0.3)--node[below]{\small $p_2 + p_3 + p_4$}(3.1,-0.3);
\end{tikzpicture}

\caption{%Example of a block schedule with blocks $B_1 = \{j_1, j_2, j_3\}$ and $B_2 = \{j_4\}$. 
Since $p_2 + p_3 + p_4 > L$, job $j_4$ has to be scheduled in a new block.}
\label{fig:pjLpj}
\end{figure}
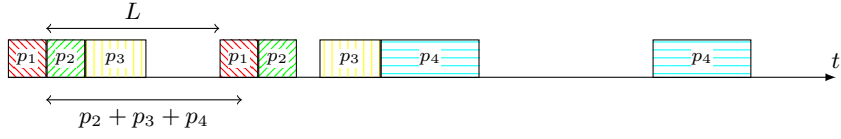

In the remaining part of the proof we compare $C_j$ and $\Coptf_j$, i.e., the completion time of job $j$ in $\Salg$ and the completion time of the $j^{\textnormal{th}}$ finishing (or starting) job in a fixed optimal schedule.
For all $j > j_2$, we will prove $C_j\leq 1.5\Coptf_j$, but this inequality does not necessarily always hold true for $j=j_2$.
However, with a more sophisticated analysis, we still manage to prove $\Calg\leq 1.5\Copt$ over the total of all jobs.

Let job $j$ be a job in some block $B_s$, where $s\geq 2$ and $j\neq j_2$. 
Then the completion time $C_j$ of $j$ can be expressed as:
\[
C_j = \sum_{i=1}^{s-1} H_i + \sum_{k=j_s}^j p_k + L + p_j.
\]
Using \autoref{eq:Hs}, we get
\[
C_j < \sum_{i=1}^{s-1} \left(\sum_{k=j_i}^{j_{i+1}-1} ( p_k + L + p_{j_{i+1}-1})\right) + \sum_{k=j_s}^j p_k + L + p_j \leq \sum_{i=1}^j p_i + \sum_{i=1}^{s-1}p_{j_{i+1}-1}+ sL + p_j.
\]
Applying Inequality~\ref{eq:boundL} once for blocks $B_1,\ldots,B_{s-1}$, we get:
\[
C_j < p_1+2\sum_{k=2}^{j_s} p_k + \sum_{i=1}^{s-1} p_{j_{i+1}-1}+ \sum_{k=j_s+1}^{j-1}p_k+ 2p_j+L.
\]
Applying it again for $B_1$, and then using $p_{j_2-1}\leq p_j$ (since $j>j_2$), we have
\begin{align*}
C_j&< p_1+3\sum_{k=2}^{j_2}p_k+2\sum_{k=j_2+1}^{j_s} p_k + \sum_{i=1}^{s-1} p_{j_{i+1}-1}+ \sum_{k=j_s+1}^{j-1}p_k+ 2p_j\\
&\leq p_1+3\sum_{k=2}^{j_2}p_k+2\sum_{k=j_2+1}^{j_s} p_k + \sum_{i=2}^{s-1} p_{j_{i+1}-1}+ \sum_{k=j_s+1}^{j-1}p_k+ 3p_j\\
&\leq 3\sum_{k=1}^j p_k.    
\end{align*}

Using \autoref{lemma:lbfinish} on this statement, we have $C_j<1.5C_j^{OPT}$ and thus,
\begin{align}
\sum_{j>j_2} C_j<1.5\sum_{j>j_2} C_j^{OPT_f}. \label{eq:final_jj2}
\end{align}

If $j$ is in $B_1$, we have $C_j=\sum_{k=1}^jp_k+L+p_j\leq C_{j}^{OPT_f}$, where the last inequality follows from \autoref{lemma:lbstart}.

%Furthermore, we have 
%\begin{align}
%C_1+L/2=2p_1+1.5L\leq %1.5C_1^{OPT}.\label{eq:boundfirstjob}
%\end{align}

Now, if $j=j_2$, we then have $C_{j_2}=H_1+2p_{j_2}+L=\sum_{k=1}^{j_2}p_k+p_{j_2-1}+p_{j_2}+2L$ (from \autoref{eq:Hs}) and $C_{j_2}^{OPT_f}\geq \sum_{k=1}^{j_2} p_k + L + p_{j_2}$ (from \autoref{lemma:lbstart}).
Thus, we have 
\begin{align}
\sum_{k=1}^{j_2} C_k&= \left(C_1+C_{j_2}\right)+\sum_{k=2}^{j_2-1}C_j\\
&\leq \left(2p_1+L+\sum_{k=1}^{j_2}p_k+p_{j_2-1}+p_{j_2}+2L\right)+\sum_{k=2}^{j_2-1}C_j^{OPT_f}\\
&\leq 1.5\left(C_1^{OPT_f}+C_{j_2}^{OPT_f}\right)+\sum_{k=2}^{j_2-1}C_j^{OPT_f}\leq 1.5\sum_{k=1}^{j_2}C_j^{OPT_f}
\end{align}

Therefore, $\Calg=\sum_{j=1}^n C_j\leq 1.5 \sum_{j=1}^n C_j^{OPT_f}=1.5\Copt$, following from the previous statement and \autoref{eq:final_jj2}.

\end{proof}

\subsection{CTP with related processing and delay times}\label{subsec:relatedpL}
In this section we consider variants where at least one of the tasks has a processing time equal to the delay time.
We first reuse \autoref{algo:A}: 

\begin{theorem}
\label{thm:pjpjpj}
\autoref{algo:A} is a factor-$1.5$ approximation for $1\|(p_j, p_j, p_j)\|\Cs$.
\end{theorem}

\begin{proof}
Due to \autoref{lemma:algAtime}, it remains to prove that $\Calg\leq 1.5\Copt$.
Observe that both first and the second tasks are in non-decreasing $p_j$ order in $\Salg$.
This means the completion time $C_j$ of job $j$ is at most $3 \sum_{i=1}^j p_i$.
Thus, we have $\Calg =\sum_{j=1}^n C_j \leq \sum_{j=1}^n \left(3 \sum_{i=1}^j p_i\right)\leq 1.5\Copt$, where the last inequality follows from \autoref{lemma:lbfinish}. 
\end{proof}
Now, consider \autoref{algo:C}.

\begin{algorithm}
\caption{}
\label{algo:C}
\textbf{Input}: a CTP instance with $L_j = b_j$ for each job $j$

\textbf{Output}: a schedule $\sigma$ for this instance
\begin{enumerate}
    \item Sort the jobs in non-decreasing $a_j + p_j$ order. In the following, let the jobs be indexed in this order.
    \item In this order, schedule the first task of each job right after the second task of the previously scheduled job has completed, and its second task according to the delay time. 
\end{enumerate}
\end{algorithm}

\begin{theorem}
\label{thm:ajpjpj}
\autoref{algo:C} is a factor-$2$ approximation for $1\|(a_j, p_j, p_j)\|\Cs$.
\end{theorem}

\begin{proof}
It is straightforward that \autoref{algo:C} runs in $\bigO(n \log n)$ time and always produces a feasible solution.
We have $\Calg \leq  \sum_{j=1}^n \sum_{i=1}^j (a_j + 2p_j)$ from \autoref{algo:C} and $\Copt \geq  \sum_{j=1}^n \sum_{i=1}^j (a_j + p_j)$ from \autoref{lemma:lbfinish}.
Therefore, the Theorem follows.
\end{proof}

If we modify the input, as well as the first step of \autoref{algo:C}, such that it takes instances of CTP with $L_j = a_j$, and sorts the jobs in non-decreasing $p_j + b_j$ order, we can approximate $1\|(p_j, p_j, b_j)\|\Cs$ with  factor of $2$. 
The proof is analogous to the proof of \autoref{thm:ajpjpj}.

\begin{theorem}\label{thm:pjpjbj}
Modified \autoref{algo:C} is a factor-$2$ approximation for $1\|(p_j, p_j, b_j)\|\Cs$.
\end{theorem}

\section{Bi-objective approximation}
\label{sec:biobj}

In this section, we give constant-factor $(\rho_1, \rho_2)$-approximations for all variants of the bi-objective $1\|(a_j,L_j,b_j)\|\{\Cmax, \Cs\}$ problem for which we gave constant factor approximations on, the $\Cs$-objective in this work.

Stein and Wein~\cite{stein97} defined two simple conditions on scheduling problems: \emph{Truncation} (deleting jobs from a valid schedule results in a valid partial schedule) and \emph{Composition} (a simple way of appending two valid partial schedules results in a valid schedule) and proved the following:

\begin{proposition}[Stein and Wein~\cite{stein97}, Corollary 3]
\label{prop:1}
For any scheduling problem satisfying the conditions Truncation and Composition, if there exists an $\alpha$-approximation algorithm for the minimization of makespan and a $\beta$-approximation algorithm for the minimization of sum of completion times, there exists an $(\alpha(1+\delta),\beta(\frac{\delta+1}{\delta}))$-algorithm for any $\delta > 0$.
\end{proposition}

Note that all considered coupled task problem variants fulfill these conditions. 
With this result in hand, we can now combine our $\beta$-approximation algorithms for the sum of completion times with previous $\alpha$-approximation algorithms for the minimization of makespan to get $(\rho_1,\rho_2)$-approximations for all approximated $\Cs$ problems.
We choose $\delta$ in such a way that the maximum of the two approximation factor $\max(\rho_1, \rho_2)$ is minimized.
This is a common choice in bi-objective optimization as the goal is to get the best balanced result for both objectives simultaneously.

We give these results in \autoref{table:biobjres}. The first column of the table specifies the specific variant of $1\|(a_j,L_j,b_j)\|\{\Cmax, \Cs\}$ to be approximated, with the variant identified by its job characteristics. 
The second column gives the $(\rho_1, \rho_2)$-approximation factor for each variant. 
As in our case $\rho_1$ always equals $\rho_2$, we just give one value in this column.
In the remaining columns we give the specific $\alpha$ and $\beta$ values used in \autoref{prop:1}, with a reference to their origin, as well as our choice of $\delta$.

The run time of this algorithm implied by \autoref{prop:1} is the sum of the run times of both the $\alpha$- and the $\beta$-approximation algorithms. As all used approximation algorithms for both $\Cmax$ and $\Cs$ problems run in polynomial time, all $(\rho_1, \rho_2)$-approximations given in \autoref{table:biobjres} can be computed in polynomial time as well.

\begin{table}
\centering
\begin{tabular}{ l | l | l l l } 
& $\rho_1= \rho_2$ & $\alpha$ & $\beta$ & $\delta$ \\
\hline
$(a, L_j, b)$ & $6.5$ & $3.5$~(\cite{ageev06}) & $3$~(Thm.~\ref{lemma:aLjbgreaterb}) & $6/7$ \\ 
$(a_j, L, b_j)$ & $6$ & $3$~(\cite{ageev2016}) & $3$~(Thm.~\ref{lemma:ajLbj}) & $1$ \\
$(a, L_j, b, b \leq a)$ & $5.5$ & $3.5$~(\cite{ageev06}) & $2$~(Thm.~\ref{lemma:aLjbsmallerb}) & $4/7$ \\
$(a_j, p_j, p_j)$ & $5.5$ & $3.5$~(\cite{ageev06}) & $2$~(Thm.~\ref{thm:ajpjpj}) & $4/7$ \\
$(p_j, p_j, b_j)$ & $5.5$ & $3.5$~(\cite{ageev06}) & $2$~(Thm.~\ref{thm:pjpjbj}) & $4/7$ \\
$(p_j, p_j, p_j)$ & $4$ & $2.5$~(\cite{ageev06}) & $1.5$~(Thm.~\ref{thm:pjpjpj}) & $3/5$ \\
$(1,L_j,1)$ & $3.25$ & $1.75$~(\cite{ageev07}) & $1.5$~(Thm.~\ref{lemma:1Lj1}) & $6/7$ \\
$(p_j,L,p_j)$ & $3$ & $1.5$~(\cite{ageev2016}) & $1.5$~(Thm.~\ref{lemma:pjLpj}) & $1$ \\
\end{tabular}
\caption{Results on bi-objective $(\rho_1, \rho_2)$-approximation for CPT variants with $\Cmax$ and $\Cs$ objectives}
\label{table:biobjres}
\end{table}

\section{Conclusion}
\label{sec:concl}
In this paper, we deal with the single machine coupled task scheduling problem, with the minimization of the total completion time as our objective function.
Our work extends the complexity results of Chen and Zhang~\cite{chen2021} and Kubiak~\cite{kubiak22} by introcuding two new $\NP$-hard variants, and provides several polynomial-time constant-factor approximation algorithms.
To do this, we were able to modify several known algorithmic concepts used in coupled task makespan minimization, but some of our proofs on approximation factors required more sophisticated ideas. 
E.g., in the proof of \autoref{lemma:pjLpj}, the original idea for the approximation factor only worked for jobs scheduled after a certain number of other jobs had already been processed, and a careful analysis of the approximation factor for the jobs before this cut-off was needed to get the result on hand.

We also give the first results on bi-objective approximation in the coupled task setting: We use a result from Stein and Wein~\cite{stein97}, together with constant-factor approximations on the makespan objective taken from the literature, to give bi-objective constant-factor approximations for the problem of both minimizing the sum of completion time and the makespan simultaneously. We do this for all variants of CTP with the sum of completion times objective that we managed to approximate with a constant factor.

Although we did manage to provide approximation algorithms for several coupled task scheduling problem variants with the sum of completion times objective, it is still unknown if there is a constant-factor approximation algorithm for a few important cases: for the most general case; for the cases where only one of the tasks has a fixed processing time; and for the $(p_j,L_j,p_j)$ variant. This stems from the fact that all algorithms presented in this work make use of some unique ordering of the jobs implied by the job characteristics on either the task lengths or delay lengths. In the aforementioned cases, there exists no such unique ordering using only task lengths or delay lengths.
Inapproximability results also are of interest for the problems presented in this paper, as they do exist for most CTP variants with the makespan objective~\cite{ageev06,ageev2016}. To our best knowledge, there are no such results for CTP with the total completion time objective. We point to these two open questions as suggestions for future research.

\paragraph{Acknowledgments}
David Fischer was supported by DFG grant MN 59/4-1.
P\'eter Gy\"orgyi was supported by the National Research, Development and Innovation Office grant no.~TKP2021-NKTA-01 and by the J\'anos Bolyai Research Scholarship of the Hungarian Academy of Sciences.
This research was supported by DAAD with funds of the Bundesministerium f\"ur Bildung und Forschung (BMBF).
We would also like to thank Matthias Mnich from Hamburg University of Technology and Krist\'of B\'erczi from Eötvös Lor\'and University for organizing the academic exchange between the authors of this paper.
The authors are grateful to Matthias Mnich and Tobias Stamm for their helpful comments and help in editing the manuscript.

\paragraph{Author Contributions}
Both authors contributed equally to this work.

\bibliographystyle{abbrvnat}
%\newpage
\bibliography{mybibliography}
\end{document}